\documentclass[12pt, twoside, reqno]{amsart}

\usepackage{amsmath}
\usepackage{amsfonts}
\usepackage{amssymb}
\usepackage{amsthm}

\usepackage[usenames,dvipsnames]{color}
\usepackage[colorlinks=true,linkcolor=Red,citecolor=Green]{hyperref}
\usepackage{graphicx}
\usepackage{cite}
\usepackage{caption}
\usepackage{mathtools}
\usepackage{comment}
\usepackage{enumitem}

\usepackage[top=2.5cm, bottom=2.5cm, left=2.5cm, right=2.5cm, headsep=0.2in]{geometry}

\DeclareMathOperator{\supp}{supp}

\DeclareMathOperator{\sgn}{sgn}

\theoremstyle{plain}
\newtheorem{theorem}{Theorem}[section]

\newtheorem{lemma}[theorem]{Lemma}
\newtheorem{rem}[theorem]{Remark}
\newtheorem{prop}[theorem]{Proposition}

\newtheorem{ex}[theorem]{Example}

\newcounter{A}
\newcounter{B}
\newcounter{C}

\newtheorem{thma}[A]{Theorem}
\newtheorem{thmb}[B]{Theorem}
\newtheorem{thmc}[C]{Theorem}

\numberwithin{equation}{section}

\newcommand{\mc}{\mathcal}
\newcommand{\ab}{=_\mathrm{ae}}
\newcommand{\Spec}{\mathrm{Spec}}
\newcommand{\Op}{\mathrm{Op}}
\newcommand{\Lapl}{\mathcal{L}}

\begin{document}

\title[Magnetic Steklov problem on surfaces]{Magnetic Steklov problem on surfaces}

\author{Mihajlo Ceki\'c}
\address{Institut f\"ur Mathematik\\
Wintherthurerstrasse 190\\
CH-8057 Z\"urich\\ Switzerland}
\email{mihajlo.cekic@math.uzh.ch}

\author{Anna Siffert}
\address{Mathematisches Institut\\
Einsteinstr. 62\\
48149 M\" unster\\
Germany}
\email{ASiffert@uni-muenster.de}
\thanks{}

\begin{abstract}
The magnetic Dirichlet-to-Neumann map encodes the voltage-to-current measurements under the influence of a magnetic field. In the case of surfaces, we provide precise spectral asymptotics expansion (up to arbitrary polynomial power) for the eigenvalues of this map. Moreover, we consider the inverse spectral problem and from the expansion we show that the spectrum of the magnetic Dirichlet-to-Neumann map, in favourable situations, uniquely determines the number and the length of boundary components, the parallel transport and the magnetic flux along boundary components. In general, we show that the situation complicates compared to the case when there is no magnetic field. For instance, there are plenty of examples where the expansion does \emph{not} detect the number of boundary components, and this phenomenon is thoroughly studied in the paper.
\end{abstract}
\maketitle

\section{Introduction}\label{sec:introduction}
In the last two decades the Steklov problem has become a very active and important research topic within the fields of geometric analysis and inverse problems \cite{Colbois-Girouard-Gordon-Sher-24}. The \emph{magnetic} Steklov problem, which also takes into account a magnetic field, has attracted increasing attention in recent years, see e.g. \cite{Liu-Tan-23, Helffer-Nicoleau-24, Colbois-Provenzano-Savo-22, Provenzano-Savo-23, Chakradhar-Gittins-Habib-Peyerimhoff-24}. In the present manuscript we study the spectral asymptotics of the magnetic Dirichlet-to-Neumann map on surfaces. We apply it to show that certain invariants of the magnetic Dirichlet-to-Neumann map can be recovered from its spectrum.

\medskip

Throughout let $(M, g)$ be a smooth, compact, orientable Riemannian surface with non-empty boundary $\partial M$. Let $A$ be a smooth purely imaginary $1$-form (modelling a magnetic potential), and let $q$ be a smooth real-valued potential function. Let $d_A := d + A$ be the covariant derivative, and denote by $d_A^*$ its formal adjoint. Consider the Schr\"odinger operator (or the magnetic Laplacian)
\[
    \Lapl_{g, A, q} := d_A^*d_A + q.
\]
It is a second order, formally self-adjoint, elliptic differential operator. In the manuscript we make the standing assumption that $0$ is not a Dirichlet eigenvalue of $\Lapl_{g, A, q}$. The \textit{Dirichlet-to-Neumann map} $\Lambda_{g, A, q}: H^{\frac{1}{2}}(\partial M) \to H^{-\frac{1}{2}}(\partial M)$ is defined by solving the Dirichlet problem (which has a unique solution by the previous assumption)
\begin{align*}
    \Lapl_{g, A, q} u &= 0, \text{ on } M,\\
    u &= f, \text{ on } \partial M,
\end{align*}
where $f\in H^{\frac{1}{2}}(\partial M)$. Then
\[
    \Lambda_{g, A, q}f = (d_A u)|_{\partial M}(\nu),
\]
where $\nu$ denotes the outer pointing boundary normal. The \emph{magnetic Steklov problem} is to study the spectrum $\Spec(\Lambda_{g, A, q})$ of $\Lambda_{g, A, q}$ (see below for definition).

\subsection{Spectral asymptotics} The operator $\Lambda_{g, A, q}$ is an elliptic, formally self-adjoint, pseudodifferential operator of order $1$ on $\partial M$, see \cite{Cekic-20}. Therefore its spectrum $\Spec(\Lambda_{g, A, q})$ is discrete, and we may enumerate it in the non-decreasing order by 
\[
    \sigma_1 \leq \sigma_2 \leq \sigma_3 \leq \dotsb,
\]
counting multiplicity. Weyl's law gives us that
\[
    \sigma_n = \frac{\pi}{\ell(\partial M)} n + \mc{O}(1), \quad n \to \infty,
\]
and so the spectrum determines $\ell(\partial M)$. In this work we are interested in much finer properties of the spectrum, and we will establish expansions into arbitrary powers of $n$ as follows.

\begin{thma}\label{thma}
    Let $(M, g)$ be a compact Riemannian surface, equipped with a smooth purely imaginary $1$-form $A$, and a smooth real-valued potential function $q$. Write  $N_{1}, \dotsc, N_m$ for the connected components of $\partial M$. Then the spectrum of $\Lambda_{g, A, q}$ is given, in the sense of multisets (i.e. allowing repetition) by
    \begin{align*}
        \Spec(\Lambda_{g, A, q}) = \bigcup_{j = 1}^m \mc{S}_j,
    \end{align*}
    where for each $j = 1, \dotsc, m$, the multiset $\mc{S}_j$ can be enumerated as $(\lambda^{(j)}_{n})_{n \in \mathbb{Z}}$, and the following asymptotic expansion holds:
    \begin{align*}
        \lambda_{n}^{(j)} &= \sum_{k = 0}^\infty b^{(j)}_{k} (1) n^{1 - k}, \quad  n \to \infty,\\
        \lambda_{n}^{(j)} &= \sum_{k = 0}^\infty b^{(j)}_{k}(-1)|n|^{1 - k}, \quad n \to -\infty,
    \end{align*}
    in the sense that for each $k_0 \in \mathbb{Z}_{\geq 0}$, we have 
    \begin{align*}
        \lambda_n^{(j)} - \sum_{k = 0}^{k_0} b_k^{(j)}(1)n^{1 - k}  &= \mathcal{O}(n^{- k_0}), \quad n \to \infty,\\
        \lambda_n^{(j)} - \sum_{k = 0}^{k_0} b_k^{(j)}(-1)|n|^{1 - k} &= \mathcal{O}(|n|^{- k_0}), \quad n \to -\infty.
    \end{align*}
    The coefficients $(b_k^{(j)}(\pm 1))_{k \in \mathbb{Z}_{\geq 0}}$ are determined by an algorithmic procedure and depend only on the full jets of $g, A, q$ on $N_{j}$. In particular, we have
    \begin{align*}
        b^{(j)}_0(\pm 1) &= \frac{2\pi}{\ell(N_j)},\\
        b^{(j)}_1(\pm 1) &= \pm \frac{2\pi}{\ell(N_j)} \left(p_{j} + \frac{1}{2\pi i} \int_{N_j} A\right),\\
        b^{(j)}_{2}(\pm 1) &= \pm \frac{i}{4\pi} \int_{N_j} \iota_{\nu} dA + \frac{1}{4\pi} \int_{N_j} q,
    \end{align*}
    where $p_1, \dotsc, p_m \in \mathbb{Z}$ are arbitrary integers.
\end{thma}

For $j = 1, \dotsc, m$, we remark that the freedom in choice of $p_j$ amounts to a translation in the spectrum, that is, if we denote by $\mu_n^{(j)}$ an enumeration of $\mc{S}_j$ with respect to the choice $p_j + q_j$, where $q_j \in \mathbb{Z}$, then it is easy to see that $\mu_n^{(j)} = \lambda_{n + q_j}^{(j)}$ for large enough $|n|$. The enumeration by $\mathbb{Z}$ is natural, as can be seen for instance by looking at the Dirichlet-to-Neumann operator of the Laplacian on the unit disk: its eigenvalues are given by $|n|$ for $n \in \mathbb{Z}$. Heuristically, this ``two-fold" property of the spectrum comes from the fact that $T^*\mathbb{S}^1$ without the zero section is disconnected. As we will see in the proof, the coefficients $b^{(j)}_k$ come up as the coefficients in the symbol expansion of a ``normal form" pseudodifferential operator to which $\Lambda_{g, A, q}$ is conjugated to. 

In the case $A = 0$, $q = 0$, the proof of the just stated theorem also shows that $b^{(j)}_k = 0$ for $k \geq 1$; this was previously shown in \cite{Edward-93, Rozenbljum-79}. We also mention that the first five heat trace invariants of $\Lambda_{g, A, q}$ were computed in \cite{Liu-Tan-23}; curiously, they do not depend on $A$. When $A = 0$, and $q$ is a non-zero \emph{constant} function, the coefficient $b_3^{(j)}$ was computed in \cite{Lagace-St-Amant-21}.

Finally, note that the presence of the factors like $\int_{N_j} A$, $\int_{N_j} \iota_\nu dA$, and $\quad \int_{N_j} q$ in the eigenvalue expansions is an instance of the broad physical phenomenon of \lq separation of spectral lines\rq\, in the presence of electric and magnetic fields, going back to \emph{Zeeman} and \emph{Stark} effects \cite{Zeeman-1896, Star-1914}; see also e.g. \cite{Fanelli-Su-Wang-Zhang-Zheng-24} and references therein. It is clear from our result that $\lambda_{n}^{(j)}$ and $\lambda_{-n}^{(j)}$ in general have different expansions as $n \to \infty$, which is a new feature compared to the case $A = 0$.

\subsection{Spectral Inverse Problem} The second question we address is to what extent does the spectrum $\Spec(\Lambda_{g, A, q})$ determine $g$, $A$, and $q$, or in other words ``can one hear" these quantities from the spectrum \cite{Kac-66}? We first state our result in the simplest non-trivial setting of a single boundary component, while the more involved discussion of multiple boundary components is postponed for later.

\subsubsection{Single boundary component} As the following result shows, from the magnetic Steklov spectrum we can uniquely recover the length of the boundary component, the flux of $A$ around $\partial M$ (modulo integers and sign), the integral of $q$ along the boundary, as well as the absolute value of the boundary integral of $\iota_\nu dA$.

\begin{thmb}\label{thmb}
    Let $(M, g)$ be a compact Riemannian surface whose boundary $\partial M$ is connected. Let $A$ be a smooth purely imaginary $1$-form, and let $q$ be a smooth real-valued potential function. Then the spectrum of $\Lambda_{g, A, q}$ uniquely determines the following quantities:
    \[
        \ell(\partial M), \quad e^{\pm \int_{\partial M} A}, \quad \left| \int_{\partial M} \iota_\nu dA\right|, \quad \int_{\partial M} q.
    \]
\end{thmb}

The fact that we may only detect $e^{\pm \int_{\partial M} A}$ and not $e^{\int_{\partial M} A}$ (and similarly for $\left|\int_{\partial M} \iota_\nu dA\right|$) is sharp. Namely, in Proposition \ref{prop:symmetry} below we will see that $\Spec(\Lambda_{g, A, q}) = \Spec(\Lambda_{g, -A, q})$ (in fact, $\Lambda_{g, A, q}$ and $\Lambda_{g, -A, q}$ are unitarily equivalent by the conjugation operator). Also, from the decomposition \eqref{eq:lambdaii} below as well as the subsequent paragraph, it follows that if we change $A$ by $-A$ just in a neighbourhood of $\partial M$, then the corresponding Dirichlet-to-Neumann maps agree up to smoothing operators, and the two spectra agree up to rapidly decaying terms of order $\mc{O}(n^{-\infty})$.

\subsubsection{Close almost bijections and spectra}\label{ssec:close-almost-bijection-intro} We now introduce some notation in order to state Theorem \ref{thmc} below. Heuristically, this notion makes precise of what it means that the spectrum (or more generally an arbitrary sequence) is close to a union of arithmetic progressions. 

Given a multiset $R = \{(a_1, b_1), \dotsc, (a_k, b_k)\}$ consisting of elements of $\mathbb{R}_{> 0} \times \mathbb{R}$, we say it \emph{generates} the multiset $\mc{S}(R)$ where
\begin{equation}\label{eq:S(R)}
    \mc{S}(R) := (a_1 \mathbb{N} + b_1) \cup \dotsb \cup (a_k \mathbb{N} + b_k),
\end{equation}
where we use the convention that $\mathbb{N} = \mathbb{Z}_{\geq 1}$. We also introduce some notation from set theory. Let $X$, $Y$, and $Z$ be unbounded multisets of real numbers, bounded from below, with no accummulation points. We say that a map $F: X \to Y$ is
\begin{itemize}
    \item \emph{close}, if $|F(x) - x| \to 0$ as $x \to \infty$;
    \item an \emph{almost bijection} if $F^{-1}(y)$ consists precisely of a single point for $y$ large enough. 
\end{itemize}
The second condition means that $F$ is a bijection up to deleting finitely many elements of $X$ and $Y$. It is a straightforward exercise to see that if $F: X \to Y$, $G: Y \to Z$ are close almost bijections, then $G \circ F: X \to Z$ is a close almost bijection, and that there is a close almost bijection $\widetilde{F}: Y \to X$, such that $\widetilde{F} \circ F(x) = x$ for large enough $x$ (uniquely defined up to finitely many points by the inverse $F^{-1}$). In what follows and in particular in Theorem \ref{thmc} below, we will sometimes consider multisets \emph{up to the equivalence relation} of being in a close almost bijection. 

The asymptotic expansion of the spectrum stated in Theorem \ref{thma} can be partially re-stated using the terminology of close almost bijections. In other words, the notion of close almost bijection relates to the non-negative powers of $|n|$ in the spectral asymptotics, while the finer invariants coming from the negative powers are not detected. Namely, if we consider the expansion up to $k_0 = 1$, we see that $\mc{S}(R)$ and $\Spec(\Lambda_{g, A, q})$ are in close almost bijection where (the notation comes from Theorem \ref{thma})
\[
    R := \bigcup_{j = 1}^{m} \left\{\left(\frac{2\pi}{\ell(N_j)}, \frac{2\pi}{\ell(N_j)} \left(p_j + \frac{1}{2\pi i} \int_{N_j} A\right)\right), \left(\frac{2\pi}{\ell(N_j)}, -\frac{2\pi}{\ell(N_j)}\left(p_j + \frac{1}{2\pi i} \int_{N_j} A\right)\right)\right\}.
\]
In fact, in Theorem \ref{thmc} below we give further relations between $R$ and $\Spec(\Lambda_{g, A, q})$, i.e. we describe to what extent does the equivalence class of $\Spec(\Lambda_{g, A, q})$ with respect to close almost bijection (and in particular $\Spec(\Lambda_{g, A, q})$ itself) determine $R$.

\subsubsection{Examples}\label{sssec:examples-intro} Before jumping into the general case, we illustrate the relation between close almost bijections and spectra through an explicit example; this should serve as motivation for Theorem \ref{thmc} below. Namely, let $M = \mathbb{S}^1 \times [-1, 1]$ equipped with the product Riemannian metric $g_{\mathrm{ann}}$, where we identify $\mathbb{S}^1$ with the quotient $\mathbb{R}/(2\pi \mathbb{Z})$; denote the variable on the $\mathbb{S}^1$ factor by $x$. Let $A = c dx$ where $c$ is a purely imaginary constant. When $c \not \in i \mathbb{Z}$, the spectrum is given by (see Example \ref{ex:planar-annulus} below for more details)
\[
    \Spec(\Lambda_{g_{\mathrm{ann}}, cdx, 0}) = \bigcup_{k \in \mathbb{Z}} \left\{ |k - ic| \tanh(|k - ic|), |k - ic|\coth(|k - ic|)\right\},
\]
and so using the expansions $\tanh x = 1 + \mc{O}(e^{-2x})$ and $\coth x = 1 + \mc{O}(e^{-2x})$ as $x \to \infty$, we get that $\Spec(\Lambda_{g_{\mathrm{ann}}, cdx, 0})$ is in close almost bijection with $\mc{S}(R)$, where
\[
    R = \{(1, ic), (1, -ic), (1, ic), (1, -ic)\}.
\]

Another setting where we can compute the magnetic Steklov spectrum explicitly is for the Euclidean unit disk $(\mathbb{D}^2, g_{\mathrm{Eucl}})$ equipped with an \emph{Aharonov-Bohm potential}, given with respect to coordinates $(x_1, x_2) \in \mathbb{R}^2$ as
\[
    A(x_1, x_2) = \frac{c}{x_1^2 + x_2^2}(-x_2 dx_1 + x_1 dx_2),\quad c \in i\mathbb{R}.
\]
Note that this potential is no longer smooth as it has a singularity at the origin so it does not fit into the framework of this paper. By \cite[Theorem 33]{Colbois-Provenzano-Savo-22} (see also Remark \ref{rem:aharonov-bohm} below) if $ic \in (0, \frac{1}{2}]$, we have
    \begin{equation}\label{eq:ah-bohm-spectrum}
        \Spec(\Lambda_{g_{\mathrm{Eucl}}, A, 0}) = \{|k - ic| \mid k \in \mathbb{Z}\},
    \end{equation}
    and hence $\Spec(\Lambda_{g_{\mathrm{Eucl}}, A, 0})$ is in close almost bijection with $\mc{S}(R)$ for
    \[
        R = \{(1, ic), (1, -ic)\}.
    \]

\subsubsection{General case: multiple boundary components} Finally, we state our main result in the general case addressing the spectral inverse problem. As explained in \S \ref{ssec:close-almost-bijection-intro} and \S \ref{sssec:examples-intro}, it is convenient to use the notion of close almost bijection.

\begin{thmc} \label{thmc}
 Let $(M, g)$ be a compact Riemannian surface, equipped with a smooth purely imaginary $1$-form $A$, and a smooth real-valued potential function $q$. Write $N_{1}, \dotsc, N_m$ for the connected components of $\partial M$. Write $p_j \in \mathbb{Z}$ for the unique integer such that 
 \[
    p_j + \frac{1}{2\pi i}\int_{N_j} A =: \alpha_j \in [0, 1), \quad j = 1, \dotsc, m,
 \]
and write
    \[
        R^\pm := R^\pm(M, g, A, q) := \bigcup_{j = 1}^{m} \left\{\left(\frac{2\pi}{\ell(N_j)}, \pm\frac{2\pi}{\ell(N_j)} \alpha_j\right)\right\}.
    \]
    
\begin{enumerate}[itemsep=5pt, label*=\arabic*.]
    \item The equivalence class of $\Spec(\Lambda_{g, A, q})$ with respect to the relation of close almost bijection uniquely determines $\mc{S}(R^+ \cup R^-)$ up to finitely many elements. More precisely, let $g'$ be another Riemannian metric on $M$, $A'$ a smooth purely imaginary $1$-form, and $q'$ is a smooth real-valued potential function, and write
    \[
        (R')^\pm := R^\pm(M, g', A', q').
    \]
    Then, there exists a close almost bijection 
    \[
        F: \Spec(\Lambda_{g, A, q}) \to \Spec(\Lambda_{g', A', q'}),
    \]
    if and only if $\mc{S}(R^+ \cup R^-)$ is equal to $\mc{S}((R')^+ \cup (R')^-)$ up to removing finitely many elements from both multisets.
 
    \item Assume that $\left(\ell(N_j)\right)_{j = 1}^m$ are all distinct, and $\alpha_{j} \not \in \{\frac{1}{4}, \frac{3}{4}\}$ for $j = 1, \dotsc, m$. Then the equivalence class of $\Spec(\Lambda_{g, A, q})$ with respect to the relation of close almost bijection uniquely determines $m$, $\left(\ell(N_j)\right)_{j = 1}^m$, and $\left(e^{\pm \int_{N_j} A}\right)_{j = 1}^m$.
\end{enumerate}
\end{thmc}

We note that the assumptions in Item 2 are sharp. In fact, if we consider for $i = 1, 2$, $(M_i, g_i)$, $A_i$, and $q_i$, as above, such that
\[
    R_1^\pm := R^\pm(M_1, g_1, A_1, q_1) := \{(1, 0)\},\quad R_2^\pm := R^\pm(M_2, g_2, A_2, q_2) := \{(3, 0), (3, \pm 1), (3, \pm 2)\},
\]
then $\mc{S}(R_1^+ \cup R_1^-) = \mc{S}(R_2^+ \cup R_2^-)$, and according to Item 1, $\Spec(\Lambda_{g_1, A_1, q_1})$ and $\Spec(\Lambda_{g_2, A_2, q_2})$ are in a close almost bijection. That the condition $\alpha_j \not \in \{\frac{1}{4}, \frac{3}{4}\}$ is sharp can be seen from a similar example given in \eqref{eq:sharp-1/4} below. 

In the case of a small number of boundary components, the preceding examples are completely classified in Proposition \ref{prop:rigid} below. More precisely, we give a complete classification of $\Spec(\Lambda_{g, A, q})$ up to the equivalence relation of close almost bijection, restricted to surfaces with a single boundary component, and comparing with surfaces that have one, two, or three boundary components.

As explained after Theorem \ref{thma}, the fact that we may only detect $e^{\pm \int_{N_j} A}$ and not $e^{\int_{N_j} A}$ (and similarly for $\left|\int_{N_j} \iota_\nu dA\right|$) is sharp as $\Spec(\Lambda_{g, A, q}) = \Spec(\Lambda_{g, -A, q})$ (see Proposition \ref{prop:symmetry} below). Again, from the decomposition \eqref{eq:lambdaii} as well as the subsequent paragraph, it follows that if we change $A$ by $-A$ just in a neighbourhood of a \emph{single} boundary component, then the corresponding Dirichlet-to-Neumann maps agree up to smoothing terms, and the two spectra agree up to terms of order $\mc{O}(n^{-\infty})$.

We emphasise that the statement in Items 1 and 2 of Theorem \ref{thmb}, ``the equivalence class of $\Spec(\Lambda_{g, A, q})$ with respect to the relation of close almost bijection, uniquely determines a quantity", is strictly stronger and so implies the statement ``$\Spec(\Lambda_{g, A, q})$ uniquely determines a quantity". For instance, it easily follows from Theorem \ref{thma} that changing $q$ does not affect the spectrum up to this relation. Thus, we need less information to detect $\mc{S}(R^+ \cup R^-)$ (up to finitely many elements); we will see in the proofs below that the reason for this are (elementary) number-theoretic properties of arithmetic progressions.

Also, in Proposition \ref{prop:example} we show that as soon as $A$ is constant near the boundary (in a suitable sense in boundary normal coordinates), then the spectrum of $\Lambda_{g, A, 0}$ agrees with $\mc{S}(R^+ \cup R^-)$ up to rapidly decaying terms of order $\mc{O}(n^{-\infty})$; in particular, this gives plenty of examples where the polynomial part of the spectrum cannot determine, for instance, the number of boundary components.

In the case $A = 0$ and $q = 0$, \cite{Girouard-Parnovski-Polterovich-Sher-14} show that $\Spec(\Lambda_{g, 0, 0})$ uniquely determines $m$ and $(\ell(N_j))_{j = 1}^m$. In the setting $A = 0$ and $q$ is a non-zero real-valued potential function \cite{Lagace-St-Amant-21} uniquely determine $m$, $(\ell(N_j))_{j = 1}^m$, as well as the coefficients $(b_k^{(j)})_{k \geq 0}$. In another direction, \cite{Edward-93} and \cite{Polterovich-Sher-15} show that $\Spec(\Lambda_{g, 0, 0})$ uniquely determines the unit ball among subsets of $\mathbb{R}^2$ and $\mathbb{R}^3$, respectively. See also \cite{Colbois-Provenzano-Savo-25} for a result in a related setting, where the volume and the conformal class of a surface is uniquely determined from the so called \emph{ground state spectrum} of the magnetic Laplacian.

\subsection{Further results} Along the lines of Theorem \ref{thmc}, it is possible to give more applications of the asymptotic spectral expansion; we briefly explain these. 

For the special case of zero magnetic field, the classification in the case of a small number of boundary components is given in Proposition \ref{prop:rigid-flat}. If we fix $(M, g)$ and only vary the magnetic and electric potentials $A$ and $q$, then $\Spec(\Lambda_{g, A, q})$ (up to the equivalence relation of close almost bijection) uniquely determines $\left(e^{\pm \int_{N_j} A}\right)_{j = 1}^{m}$, as follows from Lemma \ref{lemma:ap-uniqueness}, Item 4. In Proposition \ref{prop:ECS} we give conditions under which $\Spec(\Lambda_{g, A, q})$ uniquely determines $\mc{S}(R^+)$.

\subsection{Proof ideas} The proof of Theorem \ref{thma} is based on two ingredients: the computation of the full symbol of the magnetic Dirichlet-to-Neumann map \cite{Cekic-20}, and Theorem \ref{thm:sharpasymptotics} below, which computes spectral asymptotics up to rapidly decayling terms for any formally self-adjoint, elliptic, pseudodifferential operator $A$ (of non-zero order) on the circle $\mathbb{S}^1$, as a function of the full symbol of $A$. The latter result was originally shown in \cite{Rozenbljum-78} (see also \cite{Agranovich-84}), but we give a simplified proof that entirely relies on the symbolic calculus. We were partly motivated by the fact that the computation of coefficients $b_k^{(j)}$ in Theorem \ref{thma} was non-transparent using \cite{Rozenbljum-78}, and we also need to go further in the expansion. A major component of the proof is Theorem \ref{thm:normal-form}, which inductively constructs a unitary operator $K$ such that $K^{-1}AK$ has full symbol \emph{independent of the $x$-variable} in $T^*\mathbb{S}^1$.

Theorem \ref{thmb} follows from in a straightforward way from Theorem \ref{thma} by putting together the `two parts of the spectrum' ($\lambda_n^{(1)}$ for $n < 0$ and $n > 0$) into the non-decreasing sequence $(\sigma_n)_{n \geq 1}$ by studying cases. For Theorem \ref{thmc}, Item 1, one shows that $\mc{S}(R)$ and $\mc{S}(R')$ for some generating multisets $R$ and $R'$ are in close almost bijection, if and only if $\mc{S}(R)$ and $\mc{S}(R')$ agree up to finitely many elements. This is done in Lemmas \ref{lemma:N-cai} and \ref{lemma:acb-implies-equal} using an inductive procedure by comparing the arithmetic progressions in $R$ and $R'$, and by using the dynamics of circle rotations (an alternative proof of this step is suggested in \cite[Remark 2.10]{Girouard-Parnovski-Polterovich-Sher-14}). In turn, to show that $\mc{S}(R)$ uniquely determines $R$ in favourable situations is proved in Lemma \ref{lemma:ap-uniqueness}, using the `generating function' approach attributed to Newman-Mirsky \cite{Soifer-24}. Item 2 of Theorem \ref{thmc} then follows as an immediate consequence of these results. Full classification for cases with small number of boundary components is done in Proposition \ref{prop:ap-low-dim-uniqueness}, by using the mentioned generating function approach and by studying cases by hand.

\subsection{Perspectives} Using similar arguments, it should be possible to obtain spectral asymptotics for the Dirichlet-to-Neumann map of the \emph{connection Laplacian} $d_A^* d_A + Q$, where $A$ is a skew-Hermitian matrix of $1$-forms (unitary connection) on the Hermitian vector bundle $E = M \times \mathbb{C}^r$ over $M$, and $Q$ is a Hermitian matrix function. (Recall here that $d_A = d + A$ is the covariant derivative.)

Another interesting question is to uniquely determine \emph{all} of the coefficients $(b_k^{(j)})_{k \geq 0}$ in Theorem \ref{thma} from the spectrum $\Spec(\Lambda_{g, A, q})$, under suitable assumptions as in Theorem \ref{thmc}, Item 2. This could potentially be done by employing the strategy as in \cite[Proposition 6.3]{Lagace-St-Amant-21}. Also, within a suitable class, can we determine $(M, g)$ and $A$ such that $\Spec(\Lambda_{g, A, 0})$ is \emph{equal} up to finitely many terms to $\mc{S}(R)$ for some $R$ (see \cite{Edward-93} for a result in this direction)?

Finally, from the point of view of inverse problems an exciting question is to uniquely determine some information about the interior of $(M, g)$. The only result we are aware of in this direction is due to \cite{Florentin-25} in the analytic setting assuming that the boundary has an Anosov geodesic flow.

\subsection{Organization of the paper.} In \S \ref{ssec:symmetries}, we discuss some symmetries of the Dirichlet-to-Neumann map, while in \S \ref{ssec:examples} we compute the spectrum explicitly in the example of cylindrical manifolds. Section \ref{sec:structure-pdo} is devoted to the study of pseudodifferential operators on the circle and their spectral asymptotics. More precisely, \S \ref{ssec:quantization-circle} discusses quantizations on the circle, \S \ref{ssec:normal-form} shows that such operators can be conjugated to a normal form, \S \ref{ssec:spectral-asymptotics-general} shows spectral asymptotics for general operators. Symbol expansion for the Dirichlet-to-Neumann maps is recalled in \S \ref{ssec:symbol-expansion}, while Theorem \ref{thma} is proved in \S \ref{ssec:spectral-asymptotics-DN-map}. In \S \ref{ssec:S(R)-cab} we study unions of arithmetic progressions up to the relation of close almost bijection, while \S \ref{ssec:uniqueness-generating-multisets} studies unique determination of $R$ from $\mc{S}(R)$, and \S \ref{ssec:small} does the same for small examples. In \S \ref{ssec:uniqueness-single} we prove Theorem \ref{thmb}, while in \S \ref{ssec:uniqueness-general} we establish Theorem \ref{thmc}. Finally, in \S \ref{ssec:small-boundary} we discuss the classification for small number of boundary components, and in \S \ref{ssec:example-up-to-smoothing} we give an example in which we may compute all of the coefficients $b_k^{(j)}$.

\subsection{Acknowledgments.} The authors are grateful to Gabriel P. Paternain and Hanming Zhou for suggesting the problem and for helpful discussions. We also warmly thank the anonymous referee for comments that greatly improved the presentation of the article. During the course of writing this project M.C. received funding from an Ambizione grant (project number 201806) from the Swiss National Science Foundation and was supported by the Max Planck Institute in Bonn.

\section{Preliminaries}\label{sec:preliminaries}

In \S \ref{ssec:symmetries} we discuss symmetries of the magnetic Dirichlet-to-Neumann map (in short, the magnetic \emph{DN map}), while in \S \ref{ssec:examples} we provide examples for which we compute eigenvalues and eigenfunctions of $\Lambda_{g, A, q}$ explicitly.

\subsection{Symmetries of the magnetic DN map} 
\label{ssec:symmetries}

Let $M$ be a compact manifold with boundary, $E \to M$ a Hermitian vector bundle equipped with a unitary connection $A$, and $q$ a Hermitian endomorphism of $E$. That $A$ is unitary means that it is compatible with the Hermitian structure, that is
\[
    X \langle{s_1, s_2}\rangle_E = \langle{\iota_X d_A s_1, s_2}\rangle_E + \langle{s_1, \iota_X d_A s_2}\rangle_E, \quad s_1, s_2 \in C^\infty(M, E), \quad X \in C^\infty(M, TM),
\]
where $d_A$ is the covariant derivative of $A$, $\langle{\bullet, \bullet}\rangle$ denotes the Hermitian inner product on $E$, $\iota_X$ is contraction with $X$, and $C^\infty(M, E)$ is the space of smooth sections of $E$. For a unitary isomorphism $F: E \to E$, write $F^*A$ for the unitary connection defined by the covariant derivative $d_{F^*A} := F^{-1} d_A F$. Write $F^*q := F^{-1}qF$.

For further reference, we record an expression for the magnetic Laplacian $\Lapl_{g, A, q}$ in a local coordinate system $(x_i)_{i = 1}^n$. Write $g = \sum_{i, j = 1}^n g_{ij} dx_i \otimes dx_j$, $g^{ij} := (g^{-1})_{ij}$, and using a local trivialisation of $E$, where $A = \sum_{i = 1}^n A_i dx_i$, we have
\begin{align}\label{eq:lapl-formula}
    \mc{L}_{g, A, q} = -\frac{1}{\sqrt{\det g}} \sum_{i, j = 1}^n (\partial_{x_i} + A_i) g^{ij} \sqrt{\det g} (\partial_{x_j} + A_j) + q.
\end{align}

\begin{prop}\label{prop:symmetry}
For $F: E \to E$ a unitary isomorphism, we have
\begin{align*}
    \Lambda_{g, F^*A, F^*q} = F^{-1} \Lambda_{g, A, q} F.
\end{align*}
Assume now $\dim M = 2$ and $E = M \times \mathbb{C}$. If $Cf := \overline{f}$ denotes complex conjugation, and $c \in C^\infty(M, \mathbb{R}_{> 0})$ is a positive function, then
\[
    C^{-1}\Lambda_{g, A, q} C = \Lambda_{g, -A, q}, \quad \Lambda_{cg, A, q} = c^{-\frac{1}{2}} \Lambda_{g, A, cq}.
\]
\end{prop}
In the second part of the proposition we restricted to $\dim M = 2$ and $E = M \times \mathbb{C}$ for simplicity, but suitable versions of the identities hold in general.
\begin{proof}
     By definition $d_{F^*A} = F^{-1} d_A F$ and by taking adjoints we have $d_{F^*A}^* = F^{-1} d_A^* F$. It follows that $F^{-1}\mc{L}_{g, A, q} F = \mc{L}_{g, F^*A, F^*q}$. Therefore, if $f \in C^\infty(\partial M; E|_{\partial M})$ and $\mc{L}_{g, A, q} u = 0$ with $u|_{\partial M} = f$, then $\mc{L}_{g, F^*A, F^*q} F^{-1}u = 0$, $F^{-1} u|_{\partial M} = F^{-1} f$, and so 
     \[
        \Lambda_{g, F^*A, F^*q} F^{-1} f = \iota_{\nu} d_{F^*A} F^{-1} u = F^{-1} \iota_{\nu} d_A u = F^{-1} \Lambda_{g, A, q} f,
    \]
    proving the claim.
    \medskip
     
    For the second identity, since $A$ is skew-Hermitian, $A^* = \overline{A} = -A$, we get that $C^{-1}d_A C = d_{-A}$ and similarly for the adjoints. Therefore $C^{-1} \mc{L}_{g, A, q} C = \mc{L}_{g, -A, q}$. If $u$ and $f$ are as in the preceding paragraph, then $\mc{L}_{g, -A, q} Cu = 0$, $Cu|_{\partial M} = Cf$, and so 
    \[
        \Lambda_{g, -A, q} Cf = \iota_{\nu} d_{-A} C u = C \iota_{\nu} d_A u = C\Lambda_{g, A, q}f.
    \]

    Finally, note that (using the local coordinate expression for $\mc{L}_{g, A, q}$) we have $\mc{L}_{cg, A, q} = c^{-1} \mc{L}_{g, A, cq}$. If $\nu_c$ denotes the outer boundary normal with respect to $cg$, we see that $\nu_c = c^{-\frac{1}{2}}\nu$. Thus if $\mc{L}_{cg, A, q}u = 0$, $u|_{\partial M} = f$, then also $\mc{L}_{g, A, cq} u = 0$, and we get
    \[
        \Lambda_{cg, A, q}f = \iota_{\nu_c} d_A u|_{\partial M} = c^{-\frac{1}{2}} \Lambda_{g, A, cq}f.
    \]
    This completes the proof.
\end{proof}

\subsection{Examples}\label{ssec:examples} Next we discuss examples for which we can compute the eigenvalues and eigenfunctions of $\Lambda_{g, A, q}$ explicitly.
We use the short hand notation $\Lambda_{A, q}$ for $\Lambda_{g, A, q}$ whenever the metric $g$ is clear from the context. We do the same for $\mc{L}_{g, A, q}$.

\medskip

\begin{ex}[Cylinders I]\rm\label{ex:I} We start with a general construction, similar to \cite[Example 1.3.3]{Girouard-Polterovich-17}. Let $(N_0, g_0)$ be a closed Riemannian manifold, $L > 0$ a positive real number, and $\pi: N = N_0\times  [-L, L] \to N_0$ the product Riemannian manifold with metric $g$. Consider a purely imaginary $1$-form $A$, and a real-valued scalar potential $q$ on $N_0$. The spectrum of $\Lapl_{A, q}$ on $N_0$ is discrete, with smooth eigefunctions. Denote by $v_1, \dotsc, v_\ell$ an $L^2$-orthonormal basis of the zero eigenspace (possibly empty). For non-zero eigenvalues, let $(u_k)_{k = 1}^\infty$ denote an $L^2$-orthonormal eigenbasis, such that (note here $\lambda_k \neq 0$)
\[
    \Lapl_{A, q} u_k = \lambda_k u_k, \quad u_k \in C^\infty(N_0), \quad k = 1, 2, \dotsc.
\]
It is easy to check that
\begin{align}\label{eq:zeroeigenvalue-1}
    v_1, \dotso, v_\ell,
\end{align}
as well as 
\begin{equation}\label{eq:zeroeigenvalue-2}
    tv_1, \dotsc, tv_\ell,
\end{equation}
and
\begin{align}\label{eq:othereigenvalue}
    \cosh\left(\sqrt{\lambda_k}t\right) u_k(x)\quad\mathrm{and}\quad \sinh\left(\sqrt{\lambda_k}t\right)u_k(x), \quad k = 1, 2, \dotsc,
\end{align}
are in the kernel of $\Lapl_{\pi^*A, \pi^*q} = -\partial_t^2 + \Lapl_{A, q}$. Moreover, using an expansion into the eigenbasis of $\Lapl_{A, q}$ on each $N_0 \times \{t\}$, $t \in [-L, L]$ it is easy to see that $0$ is not a Dirichlet eigenvalue of $\Lapl_{\pi^*A, \pi^*q}$. It is immediate that \eqref{eq:zeroeigenvalue-1}, \eqref{eq:zeroeigenvalue-2}, and \eqref{eq:othereigenvalue}, are eigenfunctions of $\Lambda_{g, \pi^*A, \pi^*q}$ with eigenvalues 
\[
    0,\quad \tfrac{1}{L},\quad \sqrt{\lambda_k} \tanh\left(\sqrt{\lambda_k} L\right) \quad\mathrm{and}\quad \sqrt{\lambda_k} \coth\left(\sqrt{\lambda_k} L\right),
\]
respectively. They in fact form an eigenbasis of $\Lambda_{g, \pi^*A, \pi^*q}$, as can be see using again the expansion into the eigenbasis of $\Lapl_{A, q}$ on the boundary $\partial N$.

\end{ex}

\begin{ex}[Cylinders II]\label{ex:planar-annulus}\rm We now specialise Example \ref{ex:I} to $N_0 = \mathbb{S}^1 = \mathbb{R}/(2\pi \mathbb{Z})$ with the standard metric (of length $2\pi$), equipped with a purely imaginary $1$-form $A = H(x)\,dx$ and a real-valued potential $q$. We determine the spectrum of $\Lapl_{A, 0}$ on $\mathbb{S}^1$. Firstly, note that $A$ is cohomologous to $c\,dx$ (i.e. $\exists f \in C^\infty(\mathbb{S}^1, i\mathbb{R})$, such that $A = c\,dx + df$), where $c$ is given by the flux
\begin{align}\label{eq:cflux}
    c = \frac{1}{2\pi}\int_{\mathbb{S}^1} A = \frac{1}{2\pi} \int_0^{2\pi} H(x)\, dx \in i\mathbb{R}.
\end{align}
Since $df = e^{-f} d(e^{f})$, by the proof of Proposition \ref{prop:symmetry}, $\Lapl_{A, q}$ is unitarily equivalent to $\Lapl_{c\,dx, q}$ and so it suffices to consider $A = c\,dx$. By \eqref{eq:lapl-formula}, we have
\[
    \Lapl_{A, q} = d_A^*d_A = -(\partial_x + c)^2 + q = -\partial_x^2 -2c\partial_x - c^2 + q.
\]
Therefore, the spectrum of $\Lapl_{A, 0}$ is determined by the equation
\[
    (-\partial_x^2 - 2c\partial_x - c^2) u = \lambda u,\quad 0 \neq u \in C^\infty(\mathbb{S}^1).
\]
This is an ODE with constant coefficients, and an eigenbasis of $\Lapl_{A, 0}$ is given by
\[
    u_k(x) = e^{ikx}, \quad \lambda_k = (k - ic)^2, \quad k \in \mathbb{Z}.
\]
We see that $0$ is an eigenvalue if and only if $c \in i\mathbb{Z}$. Therefore a Steklov eigenbasis associated to non-zero eigenvalues, on $N = \mathbb{S}^1 \times [-L, L]$ is given by
\[
    \cosh(|k - ic|t) e^{ikx},\quad \sinh(|k - ic|t) e^{ikx},\quad k \in \mathbb{Z},
\]
and its eigenvalues are given by
\[
    |k - ic| \tanh(|k - ic| L),\quad |k - ic| \coth(|k - ic| L), \quad k \in \mathbb{Z},
\]
respectively. Since $\coth x, \tanh x \to 1$ as $x\to \infty$, this is in compliance with Theorem \ref{thm:sharpasymptotics} below. Moreover, since $\coth x = 1 + \mc{O}(e^{-2x}) = \tanh x$ as $x \to \infty$, in this example we get an even more precise eigenvalue asymptotics, with an \emph{exponential} remainder estimate. We also see that the coefficients introduced in Theorem \ref{thma} satisfy $b_k^{(1)} = b_k^{(2)} = 0$ for $k \geq 2$ in this case.
\end{ex}

\begin{rem}\rm \label{rem:aharonov-bohm}
    Another case where one can compute the spectrum of the magnetic DN map explicitly is on the unit disk with Euclidean metric and $A = f(x_1, x_2) (-x_2 dx_1 + x_1 dx_2)$, with $f$ a smooth function depending only on the radius. The case when $f$ is equal to a constant function corresponds to constant magnetic field (i.e. $dA$ is a constant multiple of $dx_1 \wedge dx_2$). The Laplacian $\mc{L}_{A, 0}$ commutes with the vector field generating rotations and so in polar coordinates $(r, \theta)$ we can assume that eigenfunctions are a product of $e^{ik\theta}$ (for $k \in \mathbb{Z}$) and a function depending only on the radius $g(r)$. The solutions to the eigenvalue problem then reduce to a well-known ODE for $g(r)$.

    As mentioned in the introduction, the case of 
    \[
        f(x_1, x_2) = \frac{c}{x_1^2 + x_2^2},\quad c \in i \mathbb{R},
    \]
    has attracted particular attention and with this choice of $f$, the magnetic potential $A$ is referred to as the \emph{Aharonov-Bohm potential} \cite{Colbois-Provenzano-Savo-22} (note that $A$ is \emph{not} smooth). The spectrum in this case was given in \eqref{eq:ah-bohm-spectrum} (see \cite[Theorem 33]{Colbois-Provenzano-Savo-22}).
\end{rem}

\section{Structure of pseudodifferential operators on the circle}\label{sec:structure-pdo}
Here we prove results on the structure of elliptic pseudodifferential operators on the circle by conjugating them to a ``normal form" (see \S \ref{ssec:normal-form}), and using this we show fine spectral asymptotics (see \S \ref{ssec:spectral-asymptotics-general}). We compute the full symbol of the magnetic DN map in \S \ref{ssec:symbol-expansion} and apply the preceding to show Theorem \ref{thma} in \S \ref{ssec:spectral-asymptotics-DN-map}. We recall some properties of quantization on the circle in \S \ref{ssec:quantization-circle}. 

We remark that the structure of pseudodifferential operators on the circle was proved originally in \cite{Rozenbljum-78} (see also \cite{Agranovich-84}), and here we provide a simplified proof based only on the symbolic calculus, which gives a unitary conjugacy to a normal form, and which also enables us to compute explicitly first few symbols.

We will use the standard theory of pseudodifferential operators, see \cite[Appendix E]{Dyatlov-Zworski-19} or \cite{Grigis-Sjoestrand-94}. We now introduce some notation in the particular setting of the unit circle $\mathbb{S}^1$. For $m \in \mathbb{R}$ we denote by $S^m(T^*\mathbb{S}^1)$ the space of \emph{classical symbols} on $\mathbb{S}^1$ of order $m$, that is those that admit an asymptotic expansion into positively homogeneous symbols. Let $\Psi^m(\mathbb{S}^1)$ denote the space of \emph{classical pseudodifferential operators} of order $m$ on $\mathbb{S}^1$, and write $\Psi^{-\infty}(\mathbb{S}^1) = \cap_{m \in \mathbb{R}} \Psi^{m}(\mathbb{S}^1)$ for the space of \emph{smoothing} operators. For $A \in \Psi^m(\mathbb{S}^1)$, denote by $\sigma_A$ its \emph{principal symbol}, which can be identified with an $m$-positively homogeneous function on $T^*\mathbb{S}^1$. We will write $\Op: S^m(T^*\mathbb{S}^1) \to \Psi^m(\mathbb{S}^1)$ for a \emph{quantisation} procedure on $\mathbb{S}^1$. In general, the choice of a quantisation procedure is arbitrary, but on $\mathbb{S}^1$ this choice can be made canonical up to smoothing operators, see Lemma \ref{lemma:unique-quantisation}.  

We emphasise that $m$ will denote order of the pseudodifferential operator in \S \ref{ssec:quantization-circle}--\ref{ssec:spectral-asymptotics-general}, and in \S \ref{ssec:spectral-asymptotics-DN-map} it will denote the number of boundary components; there is however no clash of notation.

\smallskip

\subsection{Quantization on the circle}\label{ssec:quantization-circle} In this paper, we will work exclusively with \emph{classical} pseudodifferential operators. We will view $\mathbb{S}^1$ as $\mathbb{R}/(2\pi \mathbb{Z})$, equipped with the canonical vector field $\partial_x$ and $1$-form $dx$. In what follows we will often identify $T^*\mathbb{S}^1 \cong \mathbb{S}^1 \times \mathbb{R}$ using $dx$ and in particular we will identify co-vectors $\xi$ with real numbers. Let $a \in S^m(T^*\mathbb{S}^1)$ be a symbol. Consider a smooth partition of unity $\chi_1 + \chi_2 = 1$ on $\mathbb{S}^1$, such that $\supp \chi_i \neq \mathbb{S}^1$ for $i = 1, 2$. There is a natural coordinate on a small neighbourhood $U_i$ of $\supp \chi_i$ (given by identification with a subset of $\mathbb{R}$); write $\varphi_i$ for the coordinate function. The quantisation $\Op(\bullet)$ on $\mathbb{S}^1$ is then defined as usual by
\begin{equation}\label{eq:quantization-formula-standard}
    \Op(a) := \psi_1 \varphi_1^* \Op^{\mathbb{R}}(\widetilde{\varphi}_{1*}a) \varphi_{1*} \chi_1 + \psi_2 \varphi_2^* \Op^{\mathbb{R}}(\widetilde{\varphi}_{2*}a) \varphi_{2*} \chi_2, 
\end{equation}
where for $i = 1, 2$, $\widetilde{\varphi}_{i}: T^*U_i \to T^*\varphi_i(U_i)$ denotes the lift of $\varphi_i$ to the cotangent bundle, $\psi_i \in C^\infty(\mathbb{S}^1)$ are such that $\psi_i \chi_i = \chi_i$ and $\supp \psi_i \subset U_i \neq \mathbb{S}^1$. Also, $\Op^{\mathbb{R}}(\bullet)$ denotes the usual quantisation on $\mathbb{R}$, 
\[
    \Op^{\mathbb{R}}(b)u (x) = (2\pi)^{-1} \int_{y \in \mathbb{R}} \int_{\xi \in \mathbb{R}} e^{i(x - y)\xi} b(x, \xi) u(y)\, d\xi dy, \quad u \in C_{\mathrm{comp}}^\infty(\mathbb{R}),\, b \in S^m(T^*\mathbb{R}).
\]

\begin{lemma}\label{lemma:unique-quantisation}
    Up to smoothing terms, there is a unique quantisation procedure $\Op$ on $\mathbb{S}^1$. More precisely, for any other quantisation $\Op'$ (corresponding to $\chi_i$, $\psi_i$, and $\varphi_i$ as above), for any $a \in S^m(T^*\mathbb{S}^1)$, $\Op(a) - \Op'(a) \in \Psi^{-\infty}(\mathbb{S}^1)$. In particular, for any $A \in \Psi^m(\mathbb{S}^1)$, the full symbol $a \in S^m(T^*\mathbb{S}^1)$ is well-defined. Moreover, we have $\Op(a) - A \in \Psi^{-\infty}(\mathbb{S}^1)$, where
    \[
        Au(x) := (2\pi)^{-1} \int_{\mathbb{R}} \int_0^{2\pi} e^{i(x_0 - y)\xi} a(x, \xi) u(y)\, dy d\xi, \quad u \in C^\infty(\mathbb{S}^1),
    \]
    where we identified $[0, 2\pi)$ with $\mathbb{S}^1$, and wrote $x_0 \in \mathbb{R}$ for an arbitrary real number such that $x \equiv x_0 \mod 2\pi$.
\end{lemma}
\begin{proof}
    We divide the proof into several steps.
    \medskip

    \emph{Step 1.} Let $A \in \Psi^m(\mathbb{S}^1)$ and let $K(x, y)$ be the Schwartz kernel of $A$. Let $\rho \in C_{\mathrm{comp}}(\mathbb{R})$ be a cut off function such that $\supp \rho \subset (-2\delta, 2\delta)$ and $\rho = 1$ on $(-\delta, \delta)$, where $\delta \in (0, \pi/4)$. Let $B$ the operator defined by the Schwartz kernel $K(x, y) \rho(y - x)$; as $K$ is smooth outside of the diagonal, $A - B  \in \Psi^{-\infty}(\mathbb{S}^1)$. 
    \medskip

    \emph{Step 2.} Define
    \[
        b(x, \xi) := \int_{x_0 - \pi/2}^{x_0 + \pi/2} e^{i(y - x_0) \xi} K(x_0, y) \rho(y - x_0)\, dy,
    \]
    where $x_0 \in \mathbb{R}$ is arbitrary such that $x - x_0 \in 2\pi \mathbb{Z}$, and we view $K(\bullet, \bullet)$ as a $2\pi \mathbb{Z} \times 2\pi \mathbb{Z}$-periodic distribution on $\mathbb{R} \times \mathbb{R}$. It is straightforward to check that this definition is independent of the choice of $x_0$. 

    For any (canonical) chart $\varphi: U \to \varphi(U) \subset \mathbb{R}$, and cut-off functions $\psi, \chi$ in $U$, by the definition of pseudodifferential operators, we see that $\varphi_{*}\psi B \chi \varphi^* \in \Psi^m(\varphi(U))$. Moreover, $\varphi_{*}\psi B \chi \varphi^* = \Op^{\mathbb{R}}(c)$, where $c$ is given by oscillatory testing
    \[
        c(x, \xi) = e^{-ix\xi} \varphi_{*}\psi B \chi \varphi^*e^{i\bullet \xi},
    \]
    see \cite[Theorem 3.4]{Grigis-Sjoestrand-94}. Observe that for any open set $V$ with $V \subset \{\chi = 1\} \cap \{\psi = 1\}$, we may take $\delta$ small enough such that $\widetilde{\varphi}^*c = b$ on $V$. In particular, it follows that
    \[
        \psi_i B \chi_i = \psi_i \varphi_i^* \Op^{\mathbb{R}}(\widetilde{\varphi}_{i*}b) \varphi_{i*} \chi_i, \quad i = 1, 2.
    \]
    Therefore,
    \[
        \Op(b) = \sum_{i = 1}^2 \psi_i \varphi_i^* \Op^{\mathbb{R}}(b) \varphi_{i*} \chi_i = \sum_{i = 1}^2 \psi_i B \chi_i \equiv B \mod \Psi^{-\infty}(\mathbb{S}^1),
    \]
    where in the last congruence, we used that
    \[
        B = (1 - \psi_1) B \chi_1 + \psi_1 B \chi_1 + (1 - \psi_2) B \chi_2 + \psi_2 B \chi_2,
    \]
    and that $(1 - \psi_1) B \chi_1, (1 - \psi_2)B \chi_2 \in \Psi^{-\infty}(\mathbb{S}^1)$, as by definition $\chi_i$ and $1 - \psi_i$ have disjoint supports, for $i = 1, 2$. Since $A \equiv B \mod \Psi^{-\infty}(\mathbb{S}^1)$, this proves the first claim.

    For the next claim, by definition any $A \in \Psi^m(\mathbb{S}^1)$ can, up to smoothing terms, be written as $\Op(a)$ for some $a \in S^m(T^*\mathbb{S}^1)$. If $\Op(a - a')$ is smoothing for some $a' \in S^m(T^*\mathbb{S}^1)$, taking the principal symbol we get $a - a' \in S^{m - 1}(T^*\mathbb{S}^1)$. Iterating, we get $a - a' \in S^{-\infty}(T^*\mathbb{S}^1)$. Thus, the full symbol of $A$ is well-defined modulo $S^{-\infty}(T^*\mathbb{S}^1)$.

\medskip

\emph{Step 3.} Here we prove the last identity and hence give an alternative proof to the first claim. We first note that for a canonical coordinate chart $\varphi: U \to \varphi(U) \subset \mathbb{S}^1$, where $U \subset \mathbb{S}^1$ is open, connected, with endpoints $x_1$ and $x_2$, we have
\begin{equation}\label{eq:same-chart}
    \varphi^* \Op(\widetilde{\varphi}_* a)\varphi_* u(x) = (2\pi)^{-1} \int_{\mathbb{R}} \int_{x_1}^{x_2} e^{i(\varphi(x) - \varphi(y)) \xi} a(x, \xi) u(y)\, dy d\xi, \quad u \in C_{\mathrm{comp}}^\infty(U).
\end{equation}
Notice that for any other canonical coordinate chart $\psi: U \to \psi(U)$, by connectedness we have $\psi(x) - \varphi(x)$ is constant for $x \in U$; therefore $\varphi^* \Op(\widetilde{\varphi}_* a)\varphi_* u$ is independent of $\varphi$.

We now observe that for $i = 1, 2$, the coordinate charts $\varphi_i$ can be extended to $\overline{\varphi}_i$, coordinate chart on $\mathbb{S}^1$ with one point deleted, and that we may choose this point as well. This process does not affect the operator $\Op(a)$, and by the preceding paragraph, $\overline{\varphi}_i^* \Op^{\mathbb{R}}(\widetilde{\overline{\varphi}}_{i*}a) \overline{\varphi}_{i*} = A$ for functions with support away from the deleted point. For $i = 1, 2$, let $\rho_i$ be a partition of unity, such that $\supp \rho_i \cap \{\psi_1 = 1\} \neq \{\psi_1 = 1\}$. We then observe that
\[
    A \chi_1 = \psi_1 \varphi_1^* \Op^{\mathbb{R}}(\widetilde{\varphi}_{1*}a) \varphi_{1*} \chi_1 + \rho_1 (1 - \psi_1) A \chi_1 + \rho_2 (1 - \psi_1) A \chi_1.
\]
By the preceding discussion, the $A$ in the last two terms can be replaced by $\varphi^* \Op^{\mathbb{R}}(\widetilde{\varphi}a) \varphi_*$ for some suitable $\varphi$ extending $\varphi_1$, and since $(1 - \psi_1) \chi_1 = 0$, these two terms are smoothing. Same kind of argument holds for $\chi_2$, and proves the claim.
\end{proof}

We now prove an auxiliary claim for the symbol of the adjoint (all adjoints are with respect to the canonical Lebesgue measure $dx$).

\begin{lemma}\label{lemma:adjoint}
    Let $a \in S^m(T^*\mathbb{S}^1)$ and write $\Op(a)^* \equiv \Op(a^*) \mod \Psi^{-\infty}(\mathbb{S}^1)$ for some $a^* \in S^m(T^*\mathbb{S}^1)$. Then
    \[
        a^* \equiv \overline{a} - \partial_{\xi} \partial_x \overline{a} - \frac{1}{2} \partial_{\xi}^2 \partial_x^2 \overline{a} + \dotsb + \frac{(-i)^k}{k!} \partial_{\xi}^k \partial_x^k \overline{a} + \dotsb \mod S^{-\infty}(T^*\mathbb{S}^1).
    \]
\end{lemma}
\begin{proof}
    We compute
    \begin{align*}
        &\Op(a)^* = \sum_{i = 1}^2 \psi_i \chi_i \varphi_i^* \Op^{\mathbb{R}}(\widetilde{\varphi}_{i*}a)^* \varphi_{i*} \psi_i (\chi_1 + \chi_2)\\
        &= \sum_{i = 1}^2 \psi_i \varphi_i^* \Op^{\mathbb{R}}(\widetilde{\varphi}_{i*}a)^* \varphi_{i*} \chi_i - \chi_2 \psi_1 \varphi_1^* \Op^{\mathbb{R}}(\widetilde{\varphi}_{1*}a)^* \varphi_{1*} \chi_1 - \chi_1 \psi_2 \varphi_2^* \Op^{\mathbb{R}}(\widetilde{\varphi}_{2*}a)^* \varphi_{2*} \chi_2\\ 
        & + \chi_1 \varphi_1^* \Op^{\mathbb{R}}(\widetilde{\varphi}_{1*}a)^* \varphi_{1*} \psi_1 \chi_2  + \chi_2 \varphi_2^* \Op^{\mathbb{R}}(\widetilde{\varphi}_{2*}a)^* \varphi_{2*} \psi_2 \chi_1 \\
        & \equiv \sum_{i = 1}^2 \psi_i \varphi_i^* \Op^{\mathbb{R}}(\widetilde{\varphi}_{i*}a)^* \varphi_{i*} \chi_i \mod \Psi^{-\infty}(\mathbb{S}^1),
    \end{align*}
    where in the first line we used that $\psi_i \chi_i = \chi_i$ for $i = 1, 2$, and $\chi_1 + \chi_2 = 1$, in the equality we also used $\chi_1 + \chi_2 = 1$. In the last line we observed that fourth to last and last terms, as well as third to last and second to last terms, cancel up to smoothing terms; this follows by adding the symmetric terms $\psi_1$ and $\psi_2$ into suitable places and using that $(1 - \psi_i) \chi_i = 0$ for $i = 1, 2$, by definition, as well as by using \eqref{eq:same-chart}.

    We are now left to use the usual formula for the symbol of the adjoint in the Euclidean space, see for instance \cite[Theorem 3.5]{Grigis-Sjoestrand-94}. (Alternatively, we may use the global quantisation formula from Lemma \ref{lemma:unique-quantisation} directly, and then use \cite[Theorem 3.5]{Grigis-Sjoestrand-94}.)
\end{proof}

\subsection{Normal form}\label{ssec:normal-form} We are now in shape to prove the main technical ingredient towards the proof of spectral asymptotics for pseudodifferential operators on the circle. We prove that such operators can be conjugated to a \emph{normal form}, that is they can be conjugated to operators whose full symbol does not depend on the $x$-variable in $T^*\mathbb{S}^1$. We will write $T^*\mathbb{S}^1 \setminus 0$ for $T^*\mathbb{S}^1$ without the zero section.

\begin{theorem}\label{thm:normal-form}
    Let $m \neq 0$, and assume that $A \in \Psi^{m}(\mathbb{S}^1)$ is elliptic and self-adjoint. Moreover, assume that $\sigma_A(x, \xi) = \sigma_A(x, -\xi)$ (or $\sigma_A(x, \xi) = -\sigma_A(x, -\xi)$) for $(x, \xi) \in T^* \mathbb{S}^1 \setminus 0$. Then, there exists a unitary operator $K$, such that
    \[
        K^{-1} A K \equiv \Op(b) \mod \Psi^{-\infty}(\mathbb{S}^1),
    \]
    where $b \in S^m(T^*\mathbb{S}^1)$ is an $x$-independent symbol, $b \equiv \sum_{i = 0}^\infty b_i$, where for $i \in \mathbb{Z}_{\geq 0}$, we have $b_i \in S^{m - i}(T^*\mathbb{S}^1)$, $b_i$ is real-valued and does not depend on the $x$-variable. Moreover, each $b_i$ can be computed using an algorithmic procedure, and $\Op(b_i)$ is self-adjoint up to smoothing operators. There exist a diffeomorphism $\varphi: \mathbb{S}^1 \to \mathbb{S}^1$, and an elliptic and unitary pseudodifferential operator $L \in \Psi^{0}(\mathbb{S}^1)$ such that if $J := (\varphi^{-1})'$, then $K = \varphi^* J^{-\frac{1}{2}} L$. In fact, if $A \equiv \sum_{i = 0}^\infty \Op(a_i)$ is the symbol expansion of $A$, then
    \[
        b_0(\xi) = (2\pi)^{m} \left(\int_0^{2\pi} |a_0|(x, \xi)^{-m^{-1}}\, dx\right)^{-m} \sgn a_0(x, 1), \quad \xi > 0,
    \]
    and $b_0(\xi) = b_0(-\xi)$ (or $b_0(\xi) = - b_0(-\xi)$). Furthermore, there exists $p \in \mathbb{Z}$, such that
    \[
        b_1(\xi) = m p \xi^{-1} b_0(\xi) + \frac{1}{2\pi} \int_0^{2\pi} e_1(x, \xi)\, dx, 
    \]
    where
    \[
        e_1(x, \xi) = a_1(\varphi^{-1}x, J^{-1}(x) \xi) - \frac{i}{2} m^2 \frac{(\varphi^{-1})''(x)}{J^3(x)} \xi^{-1} a_0(\varphi^{-1}x, J^{-1}(x) \xi), \quad \xi \neq 0.
    \]
    Finally, if $\varphi = \mathrm{id}$, i.e. $a_0$ is $x$-independent, we have
    \[
        b_2(\xi) = \frac{1}{2\pi} \int_{0}^{2\pi} \left(a_2(x, \xi) +  \frac{m - 1}{m} \frac{a_1(b_1 - a_1)}{a_0} (x, \xi) - \frac{1}{2}\xi^{-2} m(m - 1) a_0 \frac{\partial_{x}^2 k_1}{k_1} (x, \xi)\right)\, dx,
    \]
    where $k_1 \in S^0(T^*\mathbb{S}^1)$ is defined in \eqref{eq:k1} below as a function of $a_0$, $a_1$, and $b_1$.
\end{theorem}

\begin{proof}
    We divide the proof into six steps.
    
    \medskip
 
    \emph{Step 1: construction of $\varphi$.} If $\varphi: \mathbb{S}^1 \to \mathbb{S}^1$ is a diffeomorphism, write $d\varphi(x) (\partial_x) = \varphi'(x) \partial_x$, where without loss of generality here we view $\varphi: [0, 2\pi) \to [0, 2\pi)$ as a function with $\varphi(0) = 0$ and $\varphi(2\pi) = 2\pi$. Then $dx(\varphi(x)) \circ d\varphi(x) = \varphi'(x) dx(x)$. Write $a_0$ for the principal symbol of $A$; then by assumption $a_0(x, \xi) = a_0(x, -\xi)$ and since $A$ is elliptic, $a_0(x, 1)$ is nowhere zero.
    
    For $x \in \mathbb{S}^1$ and $\xi \in T^*_{\varphi(x)}\mathbb{S}^1$, we have
    \[
        \sigma(\varphi_* A \varphi^*)(\varphi(x), \xi) = a_0(x, \xi \circ d\varphi(x)).
    \]
    In particular, for $\xi > 0$, using positive homogeneity, assuming that $\varphi'(x) > 0$, and equating to an $x$-independent symbol $b_0$, we get
    \[
        \xi^{m} b_0(1) = a_0(x, 1) \varphi'(x)^{m} \xi^{m}.
    \]
    The preceding equation on $\mathbb{R}$ can be solved by
    \[
        \varphi(x) := \int_0^x \left(\frac{b_0(1)}{a_0(y, 1)}\right)^{m^{-1}}\, dy,
    \]
    where we assume that $b_0(1)$ is chosen to have the same sign as $a_0(x, 1)$. In fact, the process described above can be reversed, by asking that $\varphi(2\pi) = 2\pi$ and $\varphi' > 0$; then $\varphi$ defines a diffeomorphism on $\mathbb{S}^1$.
    The former condition is satisfied when
    \[
        b_0(\xi) = (2\pi)^{m} \left(\int_0^{2\pi} |a_0|(x, \xi)^{-m^{-1}}\, dx\right)^{-m} \sgn a_0(x, 1),
    \]
    where $\sgn(\bullet)$ denotes the sign of $\bullet$. This concludes the construction for $\xi > 0$. The case $\xi < 0$ follows by symmetry using the assumption $a_0(x, 1) = a_0(x, -1)$ (or $a_0(x, 1) = -a_0(x, -1)$); it follows also that $b_0(\xi) = b_0(-\xi)$ (or $b_0(\xi) = -b_0(-\xi)$). 

    We observe that the adjoint of the operator $J^{\frac{1}{2}} \varphi_*$ is $K_0 := \varphi^* J^{-\frac{1}{2}}$. Thus $K_0$ is unitary and by construction $K_0^{-1} A K_0$ has principal symbol $b_0$ independent of $x$.
    \medskip

    \emph{Step 2: second term.} We assume that $A$ is already in the form given by Step 1, i.e. 
    \[
        A \equiv \Op(a_0) + \Op(a_1) \mod \Psi^{m - 2}(\mathbb{S}^1),
    \]
    where $a_0$ is $x$-independent and $a_1 \in S^{m - 1}(T^*\mathbb{S}^1)$ is positively homogeneous. Since $A$ is self-adjoint, we have
    \[
        \Op(a_0) + \Op(a_1) \equiv \Op(a_0)^* + \Op(a_1)^* \equiv \Op(\overline{a}_0) + \Op(\overline{a}_1) \mod \Psi^{m -2}(\mathbb{S}^1),
    \]
    where in the second congruence we applied Lemma \ref{lemma:adjoint}, as well as the fact that $a_0$ is $x$-independent. This implies that $a_0$ and $a_1$ are real-valued.
    
    We look for $K_1 \in \Psi^0(\mathbb{S}^1)$ and $B \equiv \Op(a_0) + \Op(b_1) \mod \Psi^{m - 2}(\mathbb{S}^1)$, where $b_1 \in S^{m - 1}(T^*\mathbb{S}^1)$ is $x$-independent, such that
    \[
        0 = \sigma(K_1B - AK_1) = \sigma([K_1, A] + K_1(B - A)) = i H_{a_0} k_1 + k_1 (b_1 - a_1), 
    \]
    where $H_{\bullet}$ denotes the Hamiltonian vector field of $\bullet$, and $k_1$ denotes the principal symbol of $K_1$. Note that by definition and since $a_0$ is $x$-independent, 
    \[
        H_{a_0}(x, \xi) = \frac{\partial a_0(\xi)}{\partial \xi} \partial_x = m \xi^{-1} a_0(\xi) \partial_x, \quad \xi \neq 0.
    \]
    Therefore assuming $k \neq 0$
    \[
        \frac{\partial_x k_1}{k_1} = i \frac{b_1(\xi) - a_1(x, \xi)}{m \xi^{-1} a_0(\xi)},
    \]
    which admits a solution
    \begin{equation}\label{eq:k1}
        k_1(x, \xi) = \exp \left(i \int_0^x \frac{b_1(\xi) - a_1(y, \xi)}{m \xi^{-1} a_0(\xi)}\, dy \right),
    \end{equation}
    under the periodicity assumption $k_1(0, \xi) = k_1(2\pi, \xi)$, i.e. 
    \[
        b_1(\xi) = m \xi^{-1} a_0(\xi) p + \frac{1}{2\pi} \int_0^{2\pi} a_1(x, \xi)\, dx, \quad p \in \mathbb{Z}.
    \]
    We observe here that $b_1$ is real-valued, since $a_0$ and $a_1$ are real-valued. We claim that $K_1$ can be perturbed through smoothing operators into an invertible operator. By perturbation theory of elliptic operators, it suffices to show that $K_1$ has index zero. For the latter property, we notice that $K_1$ is obtained as a deformation through an elliptic family of operators
    \[
        K^t := \Op(k_t), \quad k_t(x, \xi) = \exp \left(ixp + it\left(\frac{x}{2\pi} \int_0^{2\pi} \frac{a_1(y, \xi)}{m \xi^{-1} a_0(\xi)}\, dy - \int_0^x \frac{a_1(y, \xi)}{m \xi^{-1} a_0(\xi)}\,dy\right) \right),
    \]
    for $t \in [0, 1]$, where $K^1 = K_1$. By the homotopy invariance of index, and the fact that $K^0$ is the multiplication operator by the function $x \mapsto e^{ixp}$, which is clearly of zero index, the conclusion follows.

    Next, since $k_1$ has unit norm, we notice that $K_1^* K_1 = \mathrm{id} + R_1$, where $R_1 \in \Psi^{-1}(\mathbb{S}^1)$ is self-adjoint. Since $K_1^* K_1$ is elliptic, positive, and bounded away from zero, we may set $Q_1 := (\mathrm{id} + R_1)^{-\frac{1}{2}} \in \Psi^0(\mathbb{S}^1)$. We may thus replace $K_1$ with $K_1 Q_1$, and note that by construction it has principal symbol equal to $k_1$, so $K_1^{-1} A K_1$ has the required form.

\medskip

\emph{Step 3: lower order terms.} Let $\ell \geq 2$, and assume according to Steps 1 and 2 above that $A \in \Psi^m(\mathbb{S}^1)$ is elliptic and self-adjoint, and takes the form
\begin{align*}
    A &\equiv \Op(a_0) + \Op(a_1) + \dotsb + \Op(a_{\ell - 1}) + \Op(a_\ell) \mod \Psi^{m - \ell - 1}(\mathbb{S}^1),
\end{align*}
where for $i = 0, \dotsc, \ell - 1$, $a_i \in S^{m - i}(T^*\mathbb{S}^1)$ are real-valued and $x$-independent, and $a_\ell \in S^{m - \ell}(T^*\mathbb{S}^1)$ is positively homogeneous. Since $A$ is self-adjoint, using Lemma \ref{lemma:adjoint}, as well as the fact that $a_0, \dotsc, a_{\ell - 1}$ are $x$-independent, we get that $a_\ell$ is real-valued.

We look for $B \in \Psi^m(\mathbb{S}^1)$ and $K_\ell \in \Psi^0(\mathbb{S}^1)$ of the form 
\begin{align*}
    B &\equiv \Op(a_0) + \Op(a_1) + \dotsb + \Op(a_{\ell - 1}) + \Op(b_\ell) \mod \Psi^{m - \ell - 1}(\mathbb{S}^1),\\
    K_{\ell} &\equiv \mathrm{id} + \Op(c_\ell) \mod \Psi^{-\ell}(\mathbb{S}^1),
\end{align*}
where $b_\ell$ is $x$-independent, and $c_\ell \in S^{1 - \ell}(T^*\mathbb{S}^1)$. Then $K_\ell B - AK_\ell \in \Psi^{m - \ell}(\mathbb{S}^1)$ with principal symbol
\[
    \sigma(K_\ell B - AK_\ell) = \sigma(B - A + [\Op(c), A] + \Op(c) (B - A)) = b_\ell - a_\ell + iH_{a_0} c_\ell.
\]
We equate this to zero to obtain
\[
    m \xi^{-1} a_0(\xi) \partial_x c_\ell = i (b_\ell - a_\ell)(x, \xi), \quad \xi \neq 0, x \in \mathbb{S}^1,
\]
which in turn has a solution
\[
    c_\ell(x, \xi) = \frac{i}{m \xi^{-1} a_0(x, \xi)} \int_0^x  (b_\ell(\xi) - a_\ell(y, \xi))\, dy,
\]
under the periodicity assumption $c_\ell(0, \xi) = c_\ell(2\pi, \xi)$, that is
\[
    b_\ell(\xi) = \frac{1}{2\pi} \int_0^{2\pi} a_\ell(x, \xi)\, dx.
\]
Since $a_\ell$ is real-valued, so is $a_\ell$. Using that $\ell \geq 2$ and $c_\ell \in S^{1 - \ell}(T^*\mathbb{S}^1)$, $\Op(c_\ell)$ is compact on $L^2(\mathbb{S}^1)$, and so $K_\ell$ has index zero. By perturbing with smoothing operators, if needed, it is thus possible to achieve that $K_\ell$ is invertible. By construction we may write $K_\ell^* K_\ell = \mathrm{id} + R_\ell$, where since the symbol $c_\ell$ is purely imaginary, $R_\ell \in \Psi^{-\ell}(\mathbb{S}^1)$, and we have $R_\ell$ is self-adjoint. Then similarly to the end of Step 2, we may replace $K_\ell$ with $K_\ell(\mathrm{id} + R_\ell)^{-\frac{1}{2}}$ to prove the unitary claim.
\medskip

\emph{Step 4: construction of $K$.} We would like to make sense of the infinite product
\[
    P := \prod_{i = 1}^\infty K_i = K_1 K_2 \cdots. 
\]
By construction in Step 3, for $i \geq 2$, $K_i \equiv \mathrm{id} + \Op(c_i) \mod \Psi^{ - i}(\mathbb{S}^1)$, and thus the expansions of partial products
\[
    P_j := \prod_{i = 2}^j K_i
\]
have fixed beginnings for large enough $j$. More precisely, for each $N > 0$, there exists $j_0$ such that for $j \geq j_0$, $P_j - P_{j_0} \in \Psi^{-N}(\mathbb{S}^1)$. By Borel summation, the limit $P := \lim_{j \to \infty} P_j \in \Psi^0(\mathbb{S}^1)$ is thus well-defined modulo smoothing. Since $P_j$ are unitary by construction, so is $P$ modulo smoothing, i.e. $P^*P - \mathrm{id} \in \Psi^{-\infty}(\mathbb{S}^1)$. Moreover, by using the same arguments as in Step 3, we may perturb $P$ by smoothing operators to make sure that $P$ is invertible, and also unitary. We set $K := K_0 K_1 P$ and $L:= K_1 P$. It is now straightforward to check that $K$ conjugates $A$ to the required form. 
\medskip

\emph{Step 5: computation of $b_0$ and $b_1$.} The coefficient $b_0$ is computed in Step 1; it is left to compute $b_1$. By Step 2, it suffices to prove the claimed formula for $e_1$, that is, the second symbol in the expansion of $E := J^{\frac{1}{2}}\varphi_* A \varphi^* J^{-\frac{1}{2}}$, where $F = \varphi_* A \varphi^*$ (with full symbol $f$). Using \cite[Equation (3.5)]{Grigis-Sjoestrand-94}, we get
\[
    f \equiv \sum_{k = 0}^\infty \frac{(-i)^k}{k!} \partial_\xi^k \partial_y^k \left(a(\varphi^{-1}x, G(x, y)^{-1}\xi) \frac{J(y)}{G(x, y)}\right)\Big|_{y = x},
\]
where $G(x, y) = \frac{\varphi^{-1}(x) - \varphi^{-1}(y)}{x - y}$ for $x \neq y$ and $G(x, x) = J(x)$; here we identify $\mathbb{S}^1$ with $[0, 2\pi)$. Note that $G > 0$ everywhere since $\varphi$ is a diffeomorphism. Then we read off that
\[
    f_0 = a_0(\varphi^{-1}x, J^{-1}(x) \xi),
\]
and also for $\xi \neq 0$
\begin{align*}
    &f_1 - a_1(\varphi^{-1}x, J^{-1}(x) \xi) = - i\partial_{\xi} \partial_y \left(a_0 (\varphi^{-1}x, G(x, y)^{-1} \xi) \frac{J(y)}{G(x, y)}\right)\Big|_{y = x}\\
    &= - i\partial_{\xi}\left( -m a_0 (\varphi^{-1}x, G(x, y)^{-1} \xi) \frac{\partial_y G}{G^2} + a_0(\varphi^{-1}x, G^{-1}(x, y)\xi) \frac{G \partial_y J  - J \partial_y G}{G^2}\right)\Big|_{y = x}\\
    &= +im^2 \xi^{-1} a_0(\varphi^{-1}x, J^{-1}(x) \xi) \frac{(\varphi^{-1})''(x)}{2 J^3(x)} - im \xi^{-1} a_0(\varphi^{-1}x, J^{-1}(x) \xi) \frac{(\varphi^{-1})''(x)}{2 J^2(x)},
\end{align*}
where in last two lines we used the chain rule and $\partial_{\xi} a_0 = m\xi^{-1} a_0$, and in the last line we used that
\[
    \partial_y G(x, y) = \frac{-J(y) (x -y) + (\varphi^{-1}(x) - \varphi^{-1}(y))}{(x - y)^2}, \quad x \neq y,\quad \partial_y G(x, x) = \frac{1}{2} (\varphi^{-1})''(x).
\]
By the formula for the symbol of the product (see \cite[Theorem 3.6]{Grigis-Sjoestrand-94}) we get that $e_0 = f_0$ and
\begin{align*}
    e_1(x,\xi) &= J^{\frac{1}{2}}(f_1 J^{-\frac{1}{2}} - i\partial_{\xi} f_0 \partial_x J^{-\frac{1}{2}})\\
    &= f_1 + \frac{i}{2} m \xi^{-1} a_0(\varphi^{-1}x, J^{-1}(x) \xi) \frac{(\varphi^{-1})''(x)}{J^2(x)}\\
    &= a_1(\varphi^{-1}x, J^{-1}(x) \xi) + +im^2 \xi^{-1} a_0(\varphi^{-1}x, J^{-1}(x) \xi) \frac{(\varphi^{-1})''(x)}{2 J^3(x)},
\end{align*}
as claimed.
\medskip

\emph{Step 6: computation of $b_2$.} We next compute $b_2$ in the case where $\varphi = \mathrm{id}$, i.e. $a_0$ is $x$-independent. For this, we revisit Step 2 and in particular we use its notation. We set $C := K_1 B - A K_1$, where $B = \Op(a_0 + b_1)$, $K_1 \equiv \Op(k_1)$ modulo smoothing (so it is invertible, but not unitary at this stage), and $C \in \Psi^{m - 2}(\mathbb{S}^1)$ with principal symbol $c$. We start by computing $c$. We have $c$ equal to the top term in the symbol expansion
\begin{align*}
    &\sum_{k \geq 0} \frac{(-i)^k}{k!} (\partial_{\xi}^k k_1 \partial_x^k b - \partial_{\xi}^k a \partial_x^k k_1)\\
    &\equiv k_1 (a_0 + b_1) - (a_0 + a_1 + a_2) k_1 + i \xi^{-1}(ma_0 + (m - 1)a_1) \partial_x k_1\\ 
    &+ \frac{1}{2} \xi^{-2} m(m - 1)a_0 \partial_x^2 k_1 \mod S^{m - 3}(T^*\mathbb{S}^1),
\end{align*}
where in the equivalence we used that $\partial_{\xi} k_0 = 0$ by $0$-homogeneity, and homogeneity of the symbol $a_0$, $a_1$, and $a_2$. By construction in Step 2 the leading two terms are zero so we get
\[
    c = - a_2 k_1 + i \xi^{-1} (m - 1) a_1 \partial_x k_1 + \frac{1}{2} \xi^{-2} m(m - 1) a_0 \partial_x^2 k_1.
\]
Using $B - K_1^{-1} A K_1 = K_1^{-1} C$ we therefore get 
\[
    K_1^{-1} A K_1 \equiv \Op(b_0 + b_1 - k_1^{-1}c) \mod \Psi^{m - 3}(\mathbb{S}^1).
\]
By Taylor expansion we compute
\[
    Q_1 = (\mathrm{id} + R_1)^{-\frac{1}{2}} \equiv \mathrm{id} - \frac{1}{2} R_1 + \frac{3}{8} R_1^2, \quad Q_1^{-1} = (\mathrm{id} + R_1)^{\frac{1}{2}} \equiv \mathrm{id} + \frac{1}{2} R_1 - \frac{1}{8} R_1^2 \mod \Psi^{-3}(\mathbb{S}^1). 
\]
We therefore have (note that $a_0 = b_0$ by assumption)
\begin{align*}
    &Q_1^{-1} K_1^{-1} A K_1 Q_1\\ 
    &\equiv \left(\mathrm{id} + \frac{1}{2} R_1 - \frac{1}{8} R_1^2\right) \Op(a_0 + b_1 - k_1^{-1}c) \left(\mathrm{id} - \frac{1}{2} R_1 + \frac{3}{8} R_1^2\right) \mod \Psi^{m - 3}(\mathbb{S}^1)\\
    &\equiv \Op(a_0 + b_1 - k_1^{-1}c) + \frac{1}{2} [R_1, \Op(a_0)] - \frac{1}{4} R_1 \Op(a_0) R_1 - \frac{1}{8} R_1^2 \Op(a_0)\\
    &\hspace{9cm}+ \frac{3}{8} \Op(a_0) R_1^2 \mod \Psi^{m - 3}(\mathbb{S}^1)\\
    &\equiv \Op(a_0 + b_1 - k_1^{-1}c) + \Op\left(\frac{i}{2} H_{a_0} r_1 - \frac{1}{4} r_1^2 a_0 - \frac{1}{8} r_1^2 a_0 + \frac{3}{8} r_1^2 a_0\right) \mod \Psi^{m - 3}(\mathbb{S}^1)\\
    &\equiv \Op\left(a_0 + b_1 - k_1^{-1}c + \frac{i}{2}m \xi^{-1} a_0 \partial_x r_1\right) \mod \Psi^{m - 3}(\mathbb{S}^1),
\end{align*}
where in the first two congruences we simply do not write the lower order terms, and in the third congruence we compute the principal symbols, and wrote $r_1$ for the principal symbol of $R_1$. Since the term $\frac{i}{2}m \xi^{-1} a_0 \partial_x r_1$ has a primitive, it disappears in the formula for $b_2(\xi)$, which concludes the proof.
\end{proof}

\begin{rem}\rm
    In the case of general $A$, i.e. without any symmetry assumption on the principal symbol of $A$, it is possible to construct a Fourier integral operator $\Phi_S: C^\infty(\mathbb{S}^1) \to C^\infty(\mathbb{S}^1)$ associated to a generating function $S \in C^\infty(T^*\mathbb{S}^1)$ positively homogeneous of order $1$, such that $\Phi_S^{-1} A \Phi_S$ has an $x$-independent principal symbol \cite{Rozenbljum-78}. For simplicity, and since it suffices for our purposes, we decided to work with the symmetry assumption. 
\end{rem}

\subsection{Spectral asymptotics for pseudodifferential operators on the circle} \label{ssec:spectral-asymptotics-general}
Recall that a self-adjoint operator $A$ is called semi-positive if there exists $C \in \mathbb{R}$ such that $A + C \geq 0$; also, if $A \in \Psi^m(\mathbb{S}^1)$ is elliptic, $m \geq 0$, and has positive principal symbol, then by G\aa rding's inequality $A$ is semi-positive. (The elliptic self-adjoint differential operator $-i\partial_x$ on $\mathbb{S}^1$ \emph{does not} satisfy this condition, and clearly has eigenvalues going to both $\pm \infty$ so it is not semi-positive.)

Now spectral theory and the fact that we may write $K^{-1}AK = B + R$, where $R$ is smoothing and $B$ has a symbol independent of $x$, implies the following remarkably sharp spectral asymptotics.

\begin{theorem}[Rozenblyum \cite{Rozenbljum-78}]\label{thm:sharpasymptotics}
    Let $A$, $(b_k)_{k \geq 0}$, and $b$ be as in Theorem \ref{thm:normal-form}, and assume that $m > 0$. Then, there exists an enumeration $(\lambda_n)_{n \in \mathbb{Z}}$ of eigenvalues of $A$, such that
    \begin{align*}
        \lambda_n &= \sum_{k = 0}^\infty n^{m - k} b_k(1), \quad n \to \infty,\\
        \lambda_n &= \sum_{k = 0}^\infty |n|^{m - k} b_k(-1), \quad n \to -\infty,
    \end{align*}
    in the sense that for each $k_0 \in \mathbb{Z}_{\geq 0}$, we have 
    \begin{align*}
        \lambda_n - \sum_{k = 0}^{k_0} n^{m - k} b_k(1) &= \mathcal{O}(n^{m - k_0 -1}), \quad n \to \infty,\\
        \lambda_n - \sum_{k = 0}^{k_0} |n|^{m - k} b_k(-1) &= \mathcal{O}(|n|^{m - k_0 -1}), \quad n \to -\infty.
    \end{align*}
\end{theorem}
\begin{proof}
    We divide the proof into three steps.
    \medskip

    \emph{Step 1.} We claim that for an arbitrary symbol $c \in S^m(T^*\mathbb{S}^1)$, and any $n \in \mathbb{Z}$ we have
    \[
        e^{-inx} \Op(c) e^{in\bullet} (x) = c(x, n) + r(x, n),
    \]
    where 
    \begin{equation}\label{eq:reference}
        \partial_x^k r(x, n) = \mathcal{O}(n^{-\infty}), \quad n \to \infty, \quad k \in \mathbb{Z}_{\geq 0}.
    \end{equation}
    For $i = 1, 2$, let $\overline{\psi}_i$ be a cut-off function with support slightly larger than $\supp \psi_i$ and such that $\overline{\psi}_i = 1$ on $\supp \psi_i$ (we recall $\psi_i$ was defined in the quantization formula \eqref{eq:quantization-formula-standard}). Indeed, we compute for each $N \in \mathbb{Z}_{\geq 0}$ that
    \begin{align*}
        &e^{-inx} \Op(c) e^{in\bullet} (x) = \sum_{j = 1}^2 \psi_j(x)
        e^{-inx} ((\varphi_{j*}\overline{\psi}_j) \Op^{\mathbb{R}}(\widetilde{\varphi}_{j*} c) (\varphi_{j*} \chi_j) e^{in \bullet})(\varphi_j x)\\
        &= \sum_{j = 1}^2 \left(r_{j, N} (x, n) + \sum_{k = 0}^N \frac{(-i)^k}{k!} \psi_j(x) \partial_{\xi}^k c(x, n) \partial_x^k \chi_j \right)\\
        &= c(x, n) + r(x, n),
    \end{align*}
    where in the second line we used oscillatory testing (see Theorem \cite[Theorem 3.4]{Grigis-Sjoestrand-94}), and for $j = 1, 2$, we introduced $r_{j, N}(x, n)$ which satisfies that $\partial_x^k r_{j, N}(x, n) = \mathcal{O}(n^{-N})$ as $n \to \infty$, for every $k \in \mathbb{Z}_{\geq 0}$. In the last line we used the properties of $\psi_j$ and $\chi_j$, and wrote $r(x, n) = r_{1, N}(x, n) + r_{2, N}(x, n)$ (which is independent of $N$).
    \medskip

    \emph{Step 2.} We claim that for $c \in S^m(T^*\mathbb{S}^1)$ independent of $x$, there exists a smoothing operator $R$, such that for $n \in \mathbb{Z}$
    \[
        (\Op(c) - R)e^{in \bullet}(x) = c(n) e^{inx}. 
    \]

    By Step 1, we have $\Op(c) e^{in\bullet} (x) = e^{in x} (c(n) + r(x, n))$ for $r$ satisfying \eqref{eq:reference}. We claim that the operator $R$ defined by $r(x, n)$ (using Fourier series) is smoothing. Indeed, its kernel is given by
    \[
        K_R(x, y) = (2\pi)^{-1}\sum_{n \in \mathbb{Z}} e^{in(x - y)} r(x, n),
    \]
    which converges and is smooth in $(x, y)$ thanks to \eqref{eq:reference}, proving the claim.
    \medskip

    \emph{Step 3.} Assume first that the principal symbol of $A$ is positive, implying that $A$ is semi-positive. For a semi-positive operator $\bullet$ write $(\mu_n(\bullet))_{n \geq 0}$ for the enumeration of its eigenvalues in the non-decreasing order. By Theorem \ref{thm:normal-form}, and Steps 1 and 2, we can write
    \[
        K^{-1} A K = B + R, \quad R \in \Psi^{-\infty}(\mathbb{S}^1),
    \]
    where $K$ is unitary, $B \in \Psi^m(\mathbb{S}^1)$ is defined by $B e^{in \bullet} (x) = b(n) e^{inx}$ for $n \in \mathbb{Z}$; $B$ has full symbol $b$. It is straightforward to check that $B$, and hence also $R$, are self-adjoint. Therefore
    \begin{equation}\label{eq:+smoothing}
        \mu_n(A) = \mu_n(B + R) = \mu_n(B) + \mathcal{O}(n^{-\infty}), \quad n \to \infty,
    \end{equation}
    where the last equality follows from the variational principle, see e.g. \cite[Lemma 2.1]{Girouard-Parnovski-Polterovich-Sher-14}. 
    
    Let $f: \mathbb{Z}_{\geq 0} \to \mathbb{Z}$ be a bijection defined by the property $\mu_n(B) = b(f(n))$. We claim that for some $C_1, C_2 > 0$, and $n_0 \in \mathbb{Z}_{\geq 0}$
    \begin{equation}\label{eq:asymptotics-bijection-f}
        C_1n \leq |f(n)| \leq C_2n, \quad n \in \mathbb{Z}_{\geq n_0}.
    \end{equation}
    Indeed, thanks to the asymptotics of $b$ we first observe that $n \mapsto b(n)$ and $n \mapsto b(-n)$ are increasing for $n$ large enough. Thus, again for $n$ large enough, $f(\{0, \dotsc, n\}) = \{-n_1, -n_1 + 1, \dotsc, n_2 - 1, n_2\}$ for some $n_1 = n_1(n), n_2 = n_2(n) \in \mathbb{Z}_{> 0}$ such that $n_1 + n_2 = n$. It follows that we may take $C_2 = 1$ in the claim. In fact, using the asymptotics of $b$, we see that for $i = 1, 2$, there are $c_i > 0$ such that $n_i(n) \geq c_i n$. Therefore the claim holds with $C_1 = \min(c_1, c_2)$.

    Finally, we can enumerate eigenvalues of $A$ as $\lambda_n(A) := \mu_{f^{-1}(n)}(B + R)$, where $n \in \mathbb{Z}$. Then
    \[
        \lambda_n(A) = \mu_{f^{-1}(n)}(B) + \mathcal{O}((f^{-1}(n))^{-\infty}) = b(n) + \mathcal{O}(|n|^{-\infty}), \quad |n| \to \infty,
    \]
    where in the first equality we used \eqref{eq:+smoothing}, and in the second one we used \eqref{eq:asymptotics-bijection-f}.

    This completes the proof in the case where the principal symbol $a_0$ is positive. If $a_0$ is negative, then the claim follows from considering $-A$, while if $a_0$ is both positive and negative, then we may apply the preceding to $A^2$ instead and the enumeration of eigenvalues of $A$ immediately follows. This completes the proof.
\end{proof}

\subsection{Asymptotic expansion of the Dirichlet-to-Neumann map}\label{ssec:symbol-expansion}
Here we recall the asymptotics of the full symbol of Dirichlet-to-Neumann map (proved in \cite{Cekic-20}), and explicitly compute first few terms in the special case of surfaces. 

\smallskip

We consider the boundary normal coordinates given by $(x_1, x_2)$ near a fixed boundary component $N$ of $\partial M$, where $x_2$ measures the normal distance to $N$. In other words, the metric $g$ in these coordinates takes the form
\[
    g = g_{11} dx_1^2 + dx_2^2.
\]
We assume for simplicity that $E = M \times \mathbb{C}$. Near $N$, we write $A = A_1 dx_1 + A_2 dx_2$. Using \cite[Lemma 2.3]{Cekic-20}, we may assume that $A_2 = 0$; more generally, we assume that $A_2 = 0$ simultaneously in coordinates $(x_1, x_2)$ as above for \emph{any} boundary component $N$. Further, by \cite[Proposition 3.3]{Cekic-20} we know that $\Lambda_{g, A, q} \in \Psi^1(\partial M)$ and we assume the full symbol of $\Lambda_{g, A, Q}$ on $N$ has the expansion
\[
    a(x, \xi) \sim \sum_{k = 0}^\infty a_{k}(x, \xi), \quad a_k \in S^{1 - k}(T^*N).
\]

\begin{prop}\label{prop:computation}
    In the coordinate system as above, near a fixed boundary component $N$, the following formulas hold:
    \begin{align*}
        a_0(x, \xi) &= \sqrt{g^{11}} |\xi_1|,\\
        a_1(x, \xi) &= -i \sqrt{g^{11}} A_1 \frac{\xi_1}{|\xi_1|},\\
        a_2(x, \xi) &= -\frac{i}{2} \partial_{x_2} A_1 \xi_1^{-1} + \frac{1}{2} (g^{11})^{-\frac{1}{2}} q |\xi_1|^{-1}. 
    \end{align*}
\end{prop}
\begin{proof}
    For completeness, we start by recalling a few facts from \cite[Section 3.2]{Cekic-20}. In \cite[Lemma 3.2]{Cekic-20} it is proved that there exists a pseudodifferential operator $B(x, -i\partial_{x_1})$ of order $1$ depending smoothly on $x_2$ for $x_2 \in [0, T]$ for some $T > 0$, satisfying
    \[
        \mc{L}_{g, A, q} \equiv (-i\partial_{x_2} + i E - i B) (-i\partial_{x_2} + i B),
    \]
    modulo smoothing, where $E = \frac{1}{2} \frac{\partial_{x_2} g^{11}}{g^{11}}$, and
    \begin{align*}
        \mc{L}_{g, A, q} &= -\frac{1}{\sqrt{\det g}} \sum_{i, j = 1}^2 (\partial_{x_i} + A_i) g^{ij} \sqrt{\det g} (\partial_{x_j} + A_j) + q\\
        &= (-i \partial_{x_2})^2 + i E (-i \partial_{x_2}) + \underbrace{g^{11} (-i\partial_{x_1})^2}_{=: Q_2} + \underbrace{(-i \frac{1}{2} \partial_{x_1} g^{11})(-i \partial_{x_1})}_{=:Q_1}\\
        &- 2ig^{11} A_1 (-i\partial_{x_1}) + \underbrace{\left(-\frac{1}{2} \partial_{x_1}g^{11} A_1 - g^{11} \partial_{x_1} A_1 - g^{11} A_1^2 + q\right)}_{=:G},
    \end{align*}
    and moreover, by \cite[Proposition 3.3]{Cekic-20}, we have $\Lambda_{g, A, q} \equiv -B|_{N}$ modulo smoothing. Equating the previous two equalities we get
    \[
        B^2 - EB + i [-i \partial_{x_2}, B] = Q_2 + Q_1 - 2ig^{11} A_1 (-i\partial_{x_1}) + G.
    \]
    Taking principal symbols and using the formula for the symbol of a composition we get
    \[
        \sum_{k = 0}^\infty \frac{(-i)^k}{k!} \partial_{\xi_1}^k b \partial_{x_1}^k b - Eb + \partial_{x_2} b = \underbrace{g^{11} \xi_1^2}_{ =: q_2} + \underbrace{(-i) \frac{1}{2} \partial_{x_1}g^{11} \xi_1}_{=: q_1} - 2ig^{11} A_1 \xi_1 + G,
    \]
    where $b$ is the full symbol of $B$. Expanding into homogeneous terms, writing $b \equiv \sum_{j = 0}^\infty b_{1 - k}$, where $b_{1 - k}$ is a symbol of order $1 - k$, we may define $b_1(x, \xi) := - \sqrt{g^{11}} |\xi_1|$. This gives further
    \begin{align*}
        b_0 &= \frac{1}{2\sqrt{g^{11}} |\xi_1|} \left(\partial_{x_2}b_1 - Eb_1 - q_1 + 2i g^{11} A_1 \xi_1 - i \partial_{\xi_1}b_1 \partial_{x_1}b_1\right)\\
        &= \frac{1}{2\sqrt{g^{11}} |\xi_1|} \left(-\frac{1}{2} \frac{\partial_{x_2} g^{11}}{\sqrt{g^{11}}} |\xi_1| + \frac{1}{2} (g^{11})^{-\frac{1}{2}} \partial_{x_2}g^{11} |\xi_1| + \frac{i}{2} \partial_{x_1}g^{11} \xi_1 + 2i g^{11} A_1 \xi_1\right.\\ 
        &- i \left.\left(-\sqrt{g^{11}} \frac{|\xi_1|}{\xi_1}\right)\left(-\frac{1}{2} \frac{\partial_{x_1} g^{11}}{\sqrt{g^{11}}} |\xi_1|\right)\right)\\
        &= i \sqrt{g^{11}} A_1 \frac{\xi_1}{|\xi_1|}.
    \end{align*}
    Similarly, we can read-off the next term in the expansion:
    \begin{align*}
        b_{-1} &= \frac{1}{2\sqrt{g^{11}} |\xi_1|} \left(\partial_{x_2} b_0 - E b_0 - G + \partial_{\xi_1} b_1 (-i\partial_{x_1}) b_0 + b_0^2 + \partial_{\xi_1} b_0 (-i \partial_{x_1}) b_1 - \frac{1}{2} \partial_{\xi_1}^2 b_1 \partial_{x_1}^2 b_1\right)\\
        &= \frac{1}{2\sqrt{g^{11}} |\xi_1|} \left(\partial_{x_2} (i \sqrt{g^{11}} A_1) \frac{\xi_1}{|\xi_1|} - \frac{1}{2} (g^{11})^{-1} \partial_{x_2} g^{11} i \sqrt{g^{11}} A_1 \frac{\xi_1}{|\xi_1|} + g^{11} \partial_{x_1} A_1 + \frac{1}{2} \partial_{x_1} g^{11} A_1\right.\\ 
        &\left.+ g^{11} A_1^2 - q + i \sqrt{g^{11}} \frac{\xi_1}{|\xi_1|} \partial_{x_1} (i\sqrt{g^{11}} A_1) \frac{\xi_1}{|\xi_1|} - g^{11} A_1^2\right)\\
        &= \frac{i}{2} \partial_{x_2} A_1 \xi_1^{-1} - \frac{q}{2 \sqrt{g^{11}} |\xi_1|},
    \end{align*}
    where in the second equality we used that $\partial_{\xi_1} b_0 \equiv 0$ and $\partial_{\xi_1}^2 b_1 \equiv 0$ by homogeneity. Since $a_0 = -b_1$, $a_1 = -b_0$, and $a_2 = -b_{-1}$, this completes the proof. The rest of the terms $b_{-2}, b_{-3}, \dotsc$, and so $a_3, a_4, \dotsc$ can be similarly computed, see \cite[Equation 3.8]{Cekic-20} (but in general the precise computations become more complicated, see also remark below). 
\end{proof}
\begin{rem}\rm
    Proposition \ref{prop:computation} for $A = 0$ and $q = 0$ shows in particular that in boundary normal coordinates we have $b_0 = b_{-1} = 0$, and \cite[Equation 3.8]{Cekic-20} shows that in fact $b_k = 0$ for all $k \leq 0$ and so $a (x, \xi) = \sqrt{g^{11}} |\xi_1|$. This means that $\Lambda_{g, 0, 0} -\frac{2\pi}{\ell(N)}|-i\partial_{x_1}| \in \Psi^{-\infty}(N)$ (we parameterise $N$ by $[0, 2\pi)$ and constant speed), which re-proves the results of \cite{Edward-93, Rozenbljum-79}. (Here, $|-i\partial_{x_1}| \in \Psi^{1}(N)$ should be seen through the action $|-i\partial_{x_1}| e^{in\bullet} (x_1) = |n| e^{in x_1}$ for $n \in \mathbb{Z}$.)
\end{rem}

\subsection{Spectral asymptotics of the DN map}\label{ssec:spectral-asymptotics-DN-map} We now assume that we have $m\geq 1$ connected components $N_1, \dotsc, N_m$ of $\partial M$. For each $j = 1, 2, \dotsc, m$ write $\varphi_{j} \in C^\infty(\partial M)$ for the function given by
\[
    \varphi_{j}(x)= 
\begin{cases}
    1, & \text{if } x\in \partial M_{j},\\
    0,              & \text{otherwise.}
\end{cases}
\]
We may write $C^\infty(\partial M) = \oplus_{i = 1}^m C^\infty(N_i)$ and denote 
\[
    \Lambda_{g, A, q}^{i j} = \varphi_{i} \Lambda_{g, A, q} \varphi_j, \quad i, j = 1, \dotsc, m.
\]
Then we have
\[
    \Lambda_{g, A, q} = \sum_{i, j = 1}^m \Lambda^{i j}_{g, A, q}.
\]
As $\Lambda_{g, A, q} \in \Psi^1(\partial M)$ by \cite[Proposition 3.3]{Cekic-20}, and using pseudolocality of pseudodifferential operators, $\Lambda_{g, A, q}^{i j}$ is smoothing for $i \neq j$. Therefore, we may write
\begin{align}\label{eq:lambdaii}
    \Lambda_{g, A, q} = \sum_{i = 1}^m \Lambda^{ii}_{g, A, q} + R, \qquad R := \sum_{i \neq j} \Lambda^{i j}_{g, A, q} \in \Psi^{-\infty}(\partial M).
\end{align}
Since $\Lambda_{g, A, q}$ is self-adjoint, each $\Lambda^{ii}_{g, A, q}$ is self-adjoint, and hence so is $R$. Finally, we note for future reference that $\Lambda_{g, -A, g}^{ij} = C \Lambda_{g, A, q}^{ij} C$. Indeed, by Proposition \ref{prop:symmetry}, $\Lambda_{g, -A, q} = C \Lambda_{g, A, q} C$, and it suffices to observe that the multiplication by $\varphi_i$ and $\varphi_j$ commutes with the conjugation $C$.

Write $\Spec(\bullet)$ for the spectrum of the operator $\bullet$. We can now combine Theorem \ref{thm:sharpasymptotics} and Proposition \ref{prop:computation} to get Theorem\,A:

\begin{proof}[Proof of Theorem \ref{thma}]
Since for $j = 1, \dotsc, m$ we have that $\Lambda^{jj}_{g, A, q}$ commute, and using \cite[Lemma 2.1]{Girouard-Parnovski-Polterovich-Sher-14}, the spectrum of $\Lambda_{g, A, q}$ is given by the union of spectra of $\Lambda^{jj}_{g, A, q}$, up to a rapidly decaying term. More precisely, writing $\mu_n(\bullet)$ for the $n$-th largest eigenvalue of a semi-positive operator $\bullet$, we have
\[
        \mu_n(\Lambda_{g, A, q}) = \mu_n\left(\sum_{j = 1}^m \Lambda^{jj}_{g, A, q} + R\right) = \mu_n\left(\sum_{j = 1}^m \Lambda^{jj}_{g, A, q}\right) + \mc{O}(n^{-\infty}), \quad n \to \infty.
\]
Note that 
\[
    \Spec\left(\sum_{j = 1}^m \Lambda^{jj}_{g, A, q}\right) = \bigcup_{j = 1}^m \Spec(\Lambda^{jj}_{g, A, q}).
\]
Since for each $j = 1, 2, \dotsc, m$, by Weyl asymptotics (or by Theorem \ref{thm:sharpasymptotics}) $\mu_n(\Lambda^{jj}_{g, A, q}) \sim C_j n$ as $n \to \infty$, for some $C_j > 0$, using Theorem \ref{thm:sharpasymptotics} it is straightforward to see that there are enumerations $\left(\lambda_n\left(\Lambda^{jj}_{g, A, q}\right)\right)_{n \in \mathbb{Z}}$ of eigenvalues of $\Lambda^{jj}_{g, A, q}$, as well as multisets $(\mc{S}_j)_{j = 1}^m$ and their enumerations as in the statement of this theorem, such that
\[
    \lambda_n^{(j)} = \lambda_n\left(\Lambda^{jj}_{g, A, q}\right) + \mc{O}(|n|^{-\infty}), \quad n \in \mathbb{Z}, \quad |n| \to \infty.
\]
Moreover, by Proposition \ref{prop:computation} and Theorem \ref{thm:sharpasymptotics} we may compute $b_0^{(j)}$, $b_1^{(j)}$, and $b_2^{(j)}$ as follows. Note that the principal symbol in Proposition \ref{prop:computation} of $\Lambda_{g, A, q}^{jj}$ is already $x$-independent if we parameterise the boundary by a constant times the arc-length so that the whole of $N_j$ is parameterised by $[0, 2\pi)$. Then 
\[
    \ell(N_j) = \int_0^{2\pi} \sqrt{g_{11}}\, dx_1 \implies g_{11} \equiv \frac{\ell(N_j)^2}{4\pi^2}.
\] 
Thus we may skip Step 1 in Theorem \ref{thm:sharpasymptotics} (i.e. we may use simply $\varphi = \mathrm{id}$) and for $j = 1, \dotsc, m$, directly obtain using also Steps 2 and 3 that
\begin{align*}
    b_0^{(j)}(\xi) &= \frac{2\pi}{\ell(N_j)} |\xi|\\
    b_1^{(j)}(\xi) &= p_j \frac{2\pi}{\ell(N_j)} \frac{\xi}{|\xi|} + \frac{1}{2\pi} \int_0^{2\pi} \left(-i\sqrt{g^{11}} A_1(x_1^{-1}(x), 0)\frac{\xi}{|\xi|}\right)\, dx\\
    &= \frac{\xi}{|\xi|} \left( p_j \frac{2\pi}{\ell(N_j)} - \frac{i}{\ell(N_j)} \int_{N_j} A\right),\\
    b_2^{(j)}(\xi) &= \frac{1}{2\pi} \int_0^{2\pi} \left(-\frac{i}{2} \partial_{x_2} A_1(x_1^{-1}(x), 0) \xi^{-1} + \frac{1}{2} (g^{11})^{-\frac{1}{2}} q(x_1^{-1}(x), 0) |\xi|^{-1}\right)\, dx\\
    &= -\frac{i}{4\pi} \xi^{-1} \int_{N_j} \partial_{x_2}A_1\, dx_1 + \frac{1}{4\pi} |\xi|^{-1} \int_{N_j} q\\
    &= \frac{i}{4\pi} \xi^{-1} \int_{N_j} \iota_{\nu} F_A + \frac{1}{4\pi} |\xi|^{-1} \int_{N_j} q,\\
\end{align*}
where in the third and fifth equalities we used that the volume form on $N_j$ is $\sqrt{g_{11}}\, dx_1$, in the last line that $F_A := dA = (-\partial_{x_2} A_1 + \partial_{x_1} A_2)\, dx_1 \wedge dx_2$ is the curvature of $A$, and that the outer normal $\nu$ is equal to $-\partial_{x_2}$ so 
\[
    \iota_\nu F_A = -\iota_{\partial_{x_2}} F_A = (-\partial_{x_2} A_1 + \partial_{x_1} A_2)\, dx_1,
\]
as well as that on $N_j$ we have $\partial_{x_1} A_2\, dx_1$ is exact. This completes the proof.
\end{proof}

\section{Unions of arithmetic progressions}\label{sec:ap}

This section is devoted to studying multisets $\mc{S}(R)$ formed by unions of arithmetic progressions generated by $R$; recall this notation was introduced in \eqref{eq:S(R)}. Three main results are proved: that $\mc{S}(R)$ is uniquely determined from its equivalence class up to the relation of close almost bijection (Lemma \ref{lemma:acb-implies-equal}); that in favourable situations, $\mc{S}(R)$ uniquely determines $R$ (Lemma \ref{lemma:ap-uniqueness}); and a complete classification of $\mc{S}(R)$ in terms of $R$ for small cardinality of $R$ (Proposition \ref{prop:ap-low-dim-uniqueness}). 

\subsection{Unions of arithmetic progressions up to close almost bijection.}\label{ssec:S(R)-cab} From Theorem \ref{thma} we see that the spectrum of the magnetic DN map is, up to small error, a union of arithmetic progressions. Recall that the notion of close almost bijection was introduced in Section \ref{sec:introduction}, as well as the notation $\mc{S}(R)$ (based in part on \cite{Girouard-Parnovski-Polterovich-Sher-14}). We introduce some further notation regarding generating multisets.

For $X$ and $Y$ as in Section \ref{sec:introduction}, we will write 
\[
    X \ab Y,
\]
if $X$ and $Y$ agree up to finitely many points, where the subscript in $\ab$ is interpreted as \emph{almost equal}. We will identify $(a, b)$ and $(c, d)$ if they generate the same arithmetic progression up to finitely many terms, i.e. if $a = c$ and there exists $m \in \mathbb{Z}$ such that $d = b + ma$; in other words, we can always select a unique representative $b \in [0, a)$. We say that the arithmetic progression generated by $(c, d)$ is a \emph{refinement} of the arithmetic progression generated by $(a, b)$, if there exists $m \in \mathbb{N}$ and $i \in \{0, 1, \dotsc, m - 1\}$ such that $c = ma$ and $d = b + ia$. We say that $R_2$ is a \emph{refinement} of $R_1$ if $R_2 \subset R_1'$, where $R_1'$ is obtained from $R_1$ by taking each arithmetic progression in $R_1$ and writing it as a union of arithmetic progressions with equal moduli.

We begin our study of multisets and arithmetic progressions by the following statement.
\begin{lemma}\label{lemma:N-cai}
    Assume $F: \mathbb{N} \to \mc{S}(R)$ is a close almost injection, where 
    \[
        R = \{(a_1, b_1), \dotsc, (a_k, b_k)\}.
    \]
    Then there exists $R_{\mathrm{rat}} \subset R$ whose elements are in $\mathbb{Q}^2$, such that $\mathbb{N} \cap \mc{S}(R_{\mathrm{rat}}) \ab \mathbb{N}$. Moreover, there exists a refinement $R'_{\mathrm{rat}}$ of $R_{\mathrm{rat}}$ such that $\mathbb{N} \ab \mc{S}(R'_{\mathrm{rat}})$.
\end{lemma}
\begin{proof}
    We divide the proof into two steps. We may assume that $R$ is minimal, i.e. there is no $R' \subsetneqq R$, and no $\widetilde{F}$ differing from $F$ at most at finitely many points, such that $\widetilde{F}: \mathbb{N} \to \mc{S}(R')$ is a close almost injection.
    \medskip

    \emph{Step 1.} We claim that one of $a_1, \dotsc, a_k$ is rational and for the sake of contradiction, we assume they are all irrational. On one hand, since $F$ is close, for any fixed $\varepsilon \in (0, \frac{1}{2})$ we have
    \[
        S(N, \varepsilon) := \frac{1}{N} \sharp \{n \in \{1, \dotsc, N\} \mid \exists i \in \{1, \dotsc, k\},\, \exists m \in \mathbb{N},\, |m a_i + b_i - n| < \varepsilon\} \to 1, \quad N \to \infty.
    \]
    On the other hand, we can estimate
    \begin{align*}
        S(N, \varepsilon) &\leq \sum_{i = 1}^k \frac{1}{N} \sharp \{n \in \{1, \dotsc, N\} \mid \exists m \in \mathbb{N},\, |m a_i + b_i - n| < \varepsilon\}\\
        & \leq \sum_{i = 1}^k \frac{1}{N} \sharp \{(m, n) \in \{1, \dotsc, CN\} \times \{1, \dotsc, N\} \mid |m a_i + b_i - n| < \varepsilon\}\\
        & \leq \sum_{i = 1}^k \frac{1}{N} \sharp \{m \in \{1, \dotsc, CN\} \mid \exists n \in \mathbb{N},\, |m a_i + b_i - n| < \varepsilon\}\\
        &= C \sum_{i = 1}^k \frac{1}{CN} \sum_{m = 1}^{CN} \mathbf{1}_{(-\varepsilon, \varepsilon)} \circ R_{a_i}^m(b_i)\\
        & \leq \sum_{i = 1}^k 3C\varepsilon = 3Ck \varepsilon,
    \end{align*}
    for $N$ large enough. Here in the second line the constant $C > 0$ depends on $R$, in the third line we used that for each $m$ there exists at most one $n$ such that $|m a_i + b_i - n| < \varepsilon$, while in the second to last line we introduced the notation $R_{a_i}(x) = x + a_i \mod \mathbb{Z}$ for the circle rotation by angle $a_i$, and $\mathbf{1}_{(-\varepsilon, \varepsilon)}$ denotes the indicator function of the interval $(-\varepsilon, \varepsilon)$. In the last line we used the unique ergodicity of irrational rotations \cite[Example 4.11]{Einsiedler-Ward-11}. This yields a contradiction for $\varepsilon$ small enough and $N$ large enough. 
    \medskip

    \emph{Step 2.} According to Step 1, we may assume without loss of generality that $a_1 = \frac{p_1}{q_1}$ is rational, where $p_1, q_1 \in \mathbb{Z}_{> 0}$ are coprime. Then modulo $\mathbb{Z}$, $a_1\mathbb{N} + b_1$ takes exactly $q_1$ values, and so either $\mathbb{N}$ intersects $a_1 \mathbb{N} + b_1$ or they are a positive distance apart. In the latter case, we can set $R' := R\setminus (a_1, b_1)$ and after possibly re-defining $F$ at finitely many points, we get a close almost injection $F: \mathbb{N} \to \mc{S}(R')$ which contradicts our assumption. In the former case we see that $b_1 = \frac{r_1}{q_1}$ for some $r_1 \in \mathbb{Z}$. Therefore, if $p_1' \in \mathbb{N}$ is such that $p_1 p_1' \equiv_{q_1} 1$, then $a_1\mathbb{N} + b_1$ intersects $\mathbb{N}$ at $p_1 \mathbb{N} + s_1$ up to finitely many points, where $s_1 := r_1 \frac{1 - p_1' p_1}{q_1}$. We therefore see that $F$ defines (after possibly changing $F$ at finitely many points) a close almost injection $F: \mc{S}(Q_1) \to \mc{S}(R\setminus\{(a_1, b_1)\})$, where $Q_1 := \cup_{i = 1}^{p_1 - 1} \{(p_1, s_1 + i)\}$. 
    
    If $p_1 = 1$ we may stop and the proof is complete. If not, by re-scaling and translating for each $i = 1, \dotsc, p_1 - 1$ the restriction $F: \mc{S}(\{(p_1, s_1 + i)\}) \to \mc{S}(R\setminus\{(a_1, b_1)\})$, we may thus iterate Step 1 and the discussion so far in Step 2, and without loss of generality assume that $(a_2, b_2) \in \mathbb{Q}^2$. After further refining each $\mc{S}(\{(p_1, s_1 + i)\})$ as in the previous paragraph and removing the intersection with $a_2 \mathbb{N} + b_2$, we may assume that the union of these refinements $Q_2$ over all $i$ closely almost injects via $F$ to $\mc{S}(R \setminus \{(a_1, b_1), (a_2, b_2)\}$. Note that if we write $a_2 = \frac{p_2}{q_2}$ where $p_2, q_2 \in \mathbb{Z}_{>0}$ are coprime, then the refinements in $Q_2$ will have steps $\frac{p_2}{\mathrm{gcd}(p_1, p_2)}$; the intersection of $\mc{S}(\{(p_1, s_1 + i)\})$ and $a_2 \mathbb{N} + b_2$ may be empty. Since $|R|$ is finite, this iteration will eventually stop, and we obtain a set $R_{\mathrm{rat}} \subset R$ with elements in $\mathbb{Q}^2$, such that $\mathbb{N} \cap \mc{S}(R_{\mathrm{rat}}) \ab \mathbb{N}$. 
    
    The refinement $R'_{\mathrm{rat}}$ of $R_{\mathrm{rat}}$ is obtained by taking the union of refinements of $(a_j, b_j)$ obtained by intersection with refinements of $\mathbb{N}$ in the procedure above. For instance, we would first take the refinement $(p_1, s_1) = (q_1 a_1, b_1 - r_1 p_1' a_1)$ of $(a_1, b_1)$.
\end{proof}

\begin{rem}\rm
    There is a close almost injection $F: \mathbb{N} \to \mathbb{N} \cup \mathbb{N} \sqrt{2}$ whose range does not land fully in $\mathbb{N}$. Indeed, take sequences $m_i, n_i \to \infty$ as $i \to \infty$, such that $|m_i \sqrt{2} - n_i| < i^{-1}$. Then set $F(n_i) := m_i$ and define $F$ to be the identity otherwise.
\end{rem}

We can then treat the case of close almost bijections of arithmetic progressions.

\begin{lemma}\label{lemma:acb-implies-equal}
    Let $R_1$ and $R_2$ be two generating multisets. Assume $F: \mc{S}(R_1) \to \mc{S}(R_2)$ is a close almost bijection. Then $\mc{S}(R_1) \ab \mc{S}(R_2)$. Moreover, for each $(a, b) \in R_1$, there exists a refinement $R_2(a, b)$ of $R_2$ such that $\mc{S}(\{(a, b)\}) \ab \mc{S}(R_2(a, b))$, and $\cup_{(a, b) \in R_1}R_2(a, b) = R_2'$, where $R_2'$ is a refinement of $R_2$ and $\mc{S}(R_2) \ab \mc{S}(R_2')$.
\end{lemma}
\begin{proof}
    If $(a, b) \in R_1$, by Lemma \ref{lemma:N-cai} (after translation and re-scaling so we are in the case of $(1, 0)$) we conclude that $\mc{S}(\{(a, b)\}) \cap \mc{S}(R_2) \ab \mc{S}(\{(a, b)\})$. We now have to justify why $\mc{S}(R_1)$ is a subset of $\mc{S}(R_2)$, up to finitely many points, i.e. to handle multiplicities. To do this, we define a close almost bijection $\widetilde{F}: \mc{S}(R_1 \setminus\{(a, b)\}) \to \mc{S}(R_2')$, where $R_2'$ is a refinement of $R_2$.

    For large $n$, and if $F(an + b) \neq an + b$, and for some $x \in \mc{S}(R_1)$ we have $F(x) = an + b$ (for large $n$, as $F$ is a close almost bijection, such $x$ exists and is unique), we set $\widetilde{F}(x) := F(an + b)$ and $\widetilde{F}(an + b) = an + b$. It is straightforward to check that $\widetilde{F}: \mc{S}(R_1) \to \mc{S}(R_2)$ is a close almost bijection. Moreover, from Lemma \ref{lemma:N-cai}, we conclude that there are refinements $R_2'$ and $R_{20}$ of $R_2$ such that $\mc{S}(R_2) \ab \mc{S}(R_{20})$, $R_2' \subset R_{20}$, and $\mc{S}(\{(a, b)\}) = \mc{S}(R_{20} \setminus R_2')$. By construction, $\widetilde{F}|_{\mc{S}(\{(a, b)\})}$ is the identity (up to finitely many points), so up to changing $\widetilde{F}$ at finitely many points, we see that $\widetilde{F}: \mc{S}(R_1\setminus\{(a, b)\}) \to \mc{S}(R_2')$ is a close almost bijection. Since $R_1\setminus\{(a, b)\}$ has cardinality $|R_1| - 1$, this inductive procedure will eventually stop, and this completes the proof of the first claim. 
    
    The claims about the existence of $R_2(a, b)$ immediately follow from the preceding arguments, and the fact that the inductive procedure eventually stops.
\end{proof}

\begin{rem}\rm
    Lemma \ref{lemma:acb-implies-equal} gives an alternative proof of what was suggested in \cite[Remark 2.10]{Girouard-Parnovski-Polterovich-Sher-14}.
\end{rem}

For further use, we record another lemma which is a consequence of the preceding two lemmas.

\begin{lemma}\label{lemma:rational-property}
    Assume we are in the setting of Lemma \ref{lemma:acb-implies-equal}. Then, there exists $k \in \mathbb{Z}_{\geq 1}$, such that for $i = 1, 2$ there are partitions $R_{i, 1}, \dotsc, R_{i, k}$ of $R_i$, satisfying the following property. For each $j = 1, \dotsc, k$, there exist $(x_j, y_j) \in \mathbb{R}^2$ such that for every $(a, b) \in R_{1,j}\cup R_{2, j}$, we have
    \begin{equation}\label{eq:rational-property}
        a \in x_j \mathbb{Q}, \quad b - y_j \in x_j \mathbb{Q}.
    \end{equation}
\end{lemma}
\begin{proof}
    We first observe that \eqref{eq:rational-property} holds for $(a, b)$ if and only if it holds for a refinement of $(a, b)$. Moreover, we see that $x_j$ is unique up to multiplication by a rational number, and $y_j$ up to translation by an element of $x_j \mathbb{Q}$. 

    By Lemma \ref{lemma:N-cai}, we see by the construction of $R_2(a, b)$ in Lemma \ref{lemma:acb-implies-equal}, that $\{(a, b)\} \cup R_2(a, b)$ satisfies \eqref{eq:rational-property} with $x_j = a$ and $y_j = b$. For $(a, b), (c, d) \in R_1$, we write $R_2(a, b) \sim R_2(c, d)$ if they both contain elements that are refinements of the same arithmetic progression in $R_2$. It is immediate to see that for an equivalence class $\mc{C}$ (with respect to $\sim$), we have
    \[
        \cup_{R_2(a, b) \in \mc{C}} \mc{S}(R_2(a, b)) \ab \mc{S}(R_{2, 1}),
    \]
    up to finitely many elements for some $R_{2, 1} \subset R_2$, which in turn agrees with $\mc{S}(R_{1, 1})$, where $R_{1, 1} = \{(a, b) \mid R_2(a, b) \in \mc{C}\}$. By construction elements of $R_{1, 1} \cup R_{2, 1}$ satisfy \eqref{eq:rational-property}, for some $(x_1, y_1)$, and by iterating this procedure we obtain the required partition. This finishes the proof.
\end{proof}

\subsection{Uniqueness for generating multisets.}\label{ssec:uniqueness-generating-multisets} We are now in a position to prove results about whether $\mc{S}(R)$ uniquely determines $R$. We use the generating function approach attributed to Newman-Mirsky (see \cite{Soifer-24}). In this subsection, for a generating multiset $R$, we will write
\[
    R^- := \{(a, -b) \mid (a, b) \in R\}.
\]

\begin{lemma}\label{lemma:ap-uniqueness}
    For $i = 1, 2$, let $R_i = \{(a_{i, 1}, b_{i, 1}), \dotsc, (a_{i, k_i}, b_{i, k_i}))\}$ for some $k_i \in \mathbb{N}$ be two generating multisets, such that 
    \[
        a_{i, 1} \leq a_{i, 2} \leq \dotsb \leq a_{i, k_i}.
    \]
    Then the following holds:
    \begin{enumerate}[label*=\arabic*.]
        \item Assume that the second coordinate of the elements of $R_1$ and $R_2$ is zero. If $\mc{S}(R_1)$ and $\mc{S}(R_2)$ are in a close almost bijection, then $R_1 = R_2$.

        \item Assume that
        \begin{equation}\label{eq:strict-ineq}
            a_{i, 1} < a_{i, 2} < \dotsb < a_{i, k_i}, \quad i = 1, 2.
        \end{equation}
        If $\mc{S}(R_1)$ and $\mc{S}(R_2)$ are in a close almost bijection, then $R_1 = R_2$, that is
        \begin{equation}\label{eq:conclusion1}
            k_1 = k_2, \quad a_{1, j} = a_{2, j}, \quad b_{2, j} - b_{1, j} \in a_{1, j} \mathbb{Z}, \quad j = 1, \dots, k_1.
        \end{equation}

        \item Assume \eqref{eq:strict-ineq} holds, as well as that 
        \begin{equation}\label{eq:anomaly}
            a_{i, j} \pm 4 b_{i, j} \not \in 4a_{i, j} \mathbb{Z}, \quad j = 1, \dotsc, k_i, \quad i = 1, 2.
        \end{equation}
        If $\mc{S}(R_1 \cup R_1^-)$ and $\mc{S}(R_2 \cup R_2^-)$ are in a close almost bijection, then 
        \begin{equation}\label{eq:conclusion2}
            k_1 = k_2, \quad a_{1, j} = a_{2, j}, \quad b_{2, j} \pm b_{1, j} \in a_{1, j} \mathbb{Z}, \quad j = 1, \dots, k_1.
        \end{equation}

        \item Assume $k_1 = k_2$ and $a_{1, j} = a_{2, j}$ for $j = 1, \dotsc, k_1$. If $\mc{S}(R_1 \cup R_1^-)$ and $\mc{S}(R_2 \cup R_2^-)$ are in a close almost bijection, then 
        \begin{equation*}
            b_{2, j} \pm b_{1, j} \in a_{1, j} \mathbb{Z}, \quad j = 1, \dots, k_1.
        \end{equation*}
    \end{enumerate}
\end{lemma}
\begin{proof}
    By Lemma \ref{lemma:acb-implies-equal}, we may assume that $\mc{S}(R_1) \ab \mc{S}(R_2)$. We now prove each Item separately. 
    \medskip

    \emph{Item 1.} Using the partition constructed in Lemma \ref{lemma:rational-property}, we may assume that there exists $x \in \mathbb{R}$ such that for $(a, 0) \in R_1 \cup R_2$, $a \in x\mathbb{Q}$. Therefore by re-scaling we may assume that the elements of $R_1 \cup R_2$ belong to $\mathbb{N} \times \{0\}$. Using the assumption and expansion into holomorphic series, we have
    \[
        \frac{1}{1 - z^{a_{1, 1}}} + \dotsb + \frac{1}{1 - z^{a_{1, k_1}}} = \frac{1}{1 - z^{a_{2, 1}}} + \dotsb + \frac{1}{1 - z^{a_{2, k_2}}}, \quad |z| < 1.
    \]
    Assume $a_{1, k_1} < a_{2, k_2}$ and that for $i = 1, 2$, $a_{i, k_i}$ appears with multiplicity $m_i$. Letting $z \to e^{\frac{2\pi i}{a_{2, k_2}}}$, we see that the left hand side remains bounded, while the right hand side diverges, contradiction. Similarly $a_{1, k_1} > a_{2, k_2}$ cannot hold, so $a_{1, k_1} = a_{2, k_2}$, and by the same argument $m_1 = m_2$. The claim follows upon iterating this procedure.
    \medskip

    \emph{Item 2.} Again, as in Item 1, using Lemma \ref{lemma:rational-property}, as well as translation and re-scaling, we may assume that the elements of $R_1 \cup R_2$ belong to $\mathbb{N} \times \mathbb{Z}$. We therefore obtain
    \[
        \frac{z^{b_{1, 1}}}{1 - z^{a_{1, 1}}} + \dotsb + \frac{z^{b_{1, k_1}}}{1 - z^{a_{1, k_1}}} = \frac{z^{b_{2, 1}}}{1 - z^{a_{2, 1}}} + \dotsb + \frac{z^{b_{2, k_2}}}{1 - z^{a_{2, k_2}}} + P(z), \quad |z| < 1,
    \]
    where $P(z)$ is a finite sum of terms of the form $z^K$, $K \in \mathbb{Z}$. Arguing as in Item 1, and using the assumption \eqref{eq:strict-ineq} to ensure there are no cancellations, we conclude that $a_{1, k_1} = a_{2, k_2}$ and that $z^{b_{1, k_1}} - z^{b_{2, k_2}}$ vanishes at $z = e^{\frac{2\pi i}{a_{1, k_1}}}$. Therefore also $b_{1, k_1} - b_{2, k_2} \in a_{1, k_1} \mathbb{Z}$. Iterating this procedure completes the proof.
    \medskip

    \emph{Item 3.} As in the preceding two items, using Lemma \ref{lemma:rational-property}, there is a partition $\cup_{j = 1}^k R_{i, j}$ of $R_i \cup R_i^-$ (for $i = 1, 2$), such that $\mc{S}(R_{1, j}) \ab \mc{S}(R_{2, j})$ and \eqref{eq:rational-property} holds. We might loose the structure of elements coming in pairs $(a, \pm b)$, however we have the following. If $(a, b), (c, d), (c, -d) \in R_{1, 1} \cup R_{2, 1}$, and $(a, -b) \in R_{1, j} \cup R_{2, j}$ for some $j \neq 1$, then we may consider $R_{1, 1} \cup R_{1, j}$ and $R_{2, 1} \cup R_{2, j}$, which satisfies \eqref{eq:rational-property}. By iterating this procedure, for $j = 1, \dotsc, k$ we may therefore assume that $R_{1, j} \cup R_{2, j}$ satisfy \eqref{eq:rational-property}, and that we are either in the setting of Item 2 (i.e. we cannot have $\{(a, b), (a, -b)\} \subset R_{i, j}$ if $(a, b) \in R_{i, j}$) or Item 3 (i.e. we have the symmetry $(a, b) \in R_{i, j}$ implies $\{(a, b), (a, -b)\} \subset R_{i, j}$). In the former case, we may apply Item 2 directly, so from now on we assume we are in the latter scenario. By re-scaling, we may assume that $(a, b), (c, d) \in R_1 \cup R_2$ implies that $a \in \mathbb{N}$, $b - d \in \mathbb{Z}$. Since $(a, b) \in R_1 \cup R_2$ implies $(a, -b) \in R_1 \cup R_2$, we get that $2b \in \mathbb{Z}$, so by another re-scaling we may assume that the elements of $R_1 \cup R_2$ belong to $\mathbb{N} \times \mathbb{Z}$.

    We may thus write
    \[
        \frac{z^{b_{1, 1}} + z^{-b_{1, 1}}}{1 - z^{a_{1, 1}}} + \dotsb + \frac{z^{b_{1, k_1}} + z^{-b_{1, k_1}}}{1 - z^{a_{1, k_1}}} = \frac{z^{b_{2, 1}} + z^{-b_{2, 1}}}{1 - z^{a_{2, 1}}} + \dotsb + \frac{z^{b_{2, k_2}} + z^{-b_{2, k_2}}}{1 - z^{a_{2, k_2}}} + P(z), \quad |z| < 1,
    \]
    where $P(z)$ is as in Item 2. Assume $a_{1, k_1} < a_{2, k_2}$ and let $z \to \zeta := e^{\frac{2\pi i}{a_{2, k_2}}}$. Write $a := a_{2, k_2}$ and $b := b_{2, k_2}$, and without loss of generality assume that $0 \leq b < a$. Then $z^{-b} + z^b$ vanishes at $z = \zeta$, so $\frac{4b}{a} = m \in \mathbb{Z}$. There are four cases: $m = 0, 1, 2, 3$. If $m = 0$, $\zeta^{b} + \zeta^{-b} \neq 0$, giving contradiction. If $m = 2$, then $\zeta^{2b} + 1 = 2 \neq 0$, contradiction. Therefore we are left to deal with $a = 4b$ and $3a = 4b$ cases, but these are excluded by assumption \eqref{eq:anomaly}. 
 
    The same argument for the case $a_{1, k_1} > a_{2, k_2}$, gives us that in fact $a_{1, k_1} = a_{2, k_2}$, and 
    \[
        z^{b_{1, k_1}} + z^{-b_{1, k_1}} - z^{b_{2, k_2}} - z^{-b_{2, k_2}} = (z^{b_{1, k_1}} - z^{b_{2, k_2}})(1 - z^{-(b_{1, k_1} + b_{2, k_2})})
    \]
    vanishes at $z = \zeta$. This shows immediately that either $b_{1, k_1} - b_{2, k_2} \in a_{1, k_1}\mathbb{Z}$ or $b_{1, k_1} + b_{2, k_2} \in a_{1, k_1} \mathbb{Z}$. We may therefore iterate this procedure and the main claim follows. This completes the proof.
    \medskip

    \emph{Item 4.} This follows from the discussion in Item 3 using the generating function approach, in particular its last paragraph.
\end{proof}

\begin{rem}\rm
    Lemma \ref{lemma:ap-uniqueness}, Item 1, was already shown in \cite[Lemmas 2.6 and 2.8]{Girouard-Parnovski-Polterovich-Sher-14} using a different approach. The above lemma gives a quick alternative proof of this fact.
\end{rem}

The result in Lemma \ref{lemma:ap-uniqueness}, Item 3, is sharp, in the sense of assumption \eqref{eq:anomaly}, that is if
\begin{equation}\label{eq:sharp-1/4}
    R_1 = \{(1, 0)\}, \quad R_2 = \{(2, 0), (4, 1), (6, 1), (12, 3)\},
\end{equation}
we have that $\mc{S}(R_1 \cup R_1^-) = \mc{S}(R_2 \cup R_2^-)$, however $R_1$ and $R_2$ satisfy the assumption \eqref{eq:strict-ineq}.

Next, we record a lemma for when a covering of $\mc{S}(\{(a, b)\} \cup \{(a, -b)\})$ by $\mc{S}(R \cup R^-)$ reduces to a covering of $\mc{S}(\{(a, b)\})$ by $\mc{S}(R)$.

\begin{lemma}\label{lemma:double-to-single-reduction}
    Let $R_1 = \{(a, b)\}$ and $R = \{(a_1, b_1), \dotsc, (a_k, b_k)\}$ be two generating multisets. Assume $\mc{S}(\{(a, b), (a, -b)\})$ and $\mc{S}(R \cup R^-)$ are in a close almost bijection. Assume moreover that $b, b - \frac{a}{2}, b \pm \frac{a}{4} \not \in a\mathbb{Z}$. Then, there exist $\varepsilon_1, \dotsc, \varepsilon_k \in \{\pm 1\}$, such that if $R' := \{(a_1, \varepsilon_1 b_1), \dotsc, (a_k, \varepsilon_k b_k)\}$, $\mc{S}(\{(a, b)\}) \ab \mc{S}(R')$.
\end{lemma}
\begin{proof}
    By Theorem \ref{lemma:acb-implies-equal}, we may assume $\mc{S}(\{(a, b), (a, -b)\}) \ab \mc{S}(R \cup R^-)$. By re-scaling, we may assume $a = 1$; we may also assume that $b \in [0, 1)$, and $b \neq 0, \frac{1}{4}, \frac{1}{2}, \frac{3}{4}$ by assumption. 
    
    By Lemma \ref{lemma:N-cai}, for $j = 1, \dotsc, k$, we have $a_j = \frac{p_j}{q_j} \in \mathbb{Q}$, where $p_j, q_j \in \mathbb{Z}_{>0}$ are coprime. We now fix $j$, and so for all $r \in \mathbb{Z}$ 
    \[
        \frac{r p_j}{q_j} + b_j \equiv \varepsilon(r) b \mod \mathbb{Z},
    \]
    for some $\varepsilon(r) \in \{\pm 1\}$. We now distinguish cases according to the value of $q_j$.
    \medskip
    
    \emph{Case $q_j \geq 3$.} Then among $\varepsilon(0), \varepsilon(1), \varepsilon(2)$ two have the same sign, and thus $\frac{i}{q_j} \equiv 0 \mod \mathbb{Z}$ for some $i \in \{1, 2\}$, contradiction.
    \medskip
    
    \emph{Case $q_j = 2$.} Then we get
    \[
            b_j \equiv \varepsilon(0) b, \quad \frac{p_j}{2} + b_j \equiv \varepsilon(1) b \mod \mathbb{Z},
    \]
    where $\varepsilon(0)$ and $\varepsilon(1)$ have different signs. It follows that
    \[
        \frac{p_j}{2} \equiv \pm 2b \mod \mathbb{Z},
    \]
    and so that $b \pm \frac{1}{4} \in \mathbb{Z}$, contradicting our assumption.
    \medskip
    
    We therefore may assume that for all $j = 1, \dotsc, k$, $q_j = 1$, and that $b_j \equiv \varepsilon_j b \mod \mathbb{Z}$ for some $\varepsilon_j \in \{\pm 1\}$. By assumption, $b \not \equiv 0, \frac{1}{2} \mod \mathbb{Z}$, so $\mathbb{N} + b$ and $\mathbb{N} - b$ are disjoint, and also $a_j \mathbb{N} + b_j$ intersects precisely one of $\mathbb{N} \pm b$. This immediately implies the conclusion and completes the proof.
\end{proof}

The condition that $b \pm \frac{a}{4} \not \in a \mathbb{Z}$ in the preceding lemma is optimal, as can be seen from the following example
\[
    R_1 = \left\{\left(1, \frac{1}{4}\right)\right\}, \quad R_2 = \left\{\left(\frac{3}{2}, \frac{1}{4}\right), \left(3, \frac{3}{4}\right)\right\}.
\]
Indeed, we have $\mc{S}(R_1 \cup R_1^-) = \mc{S}(R_2 \cup R_2^-)$, but to cover $\mathbb{N} + \frac{1}{4}$, we would need to use $\frac{3}{2}\mathbb{N} + \frac{1}{4}$ as it contains the refinement $3\mathbb{N}+ \frac{1}{4}$, but this is impossible because $\frac{3}{2}\mathbb{N} + \frac{1}{4}$ also contains as a subset the refinement $3\mathbb{N} + \frac{7}{4}$.

\begin{rem}\rm
At the time of writing of this article, it is not clear whether the condition $b, b - \frac{a}{2} \in a \mathbb{Z}$ is optimal or not, i.e. in this situation to find an example where the claim of Lemma \ref{lemma:double-to-single-reduction} is false.
\end{rem}

\subsection{Uniqueness for small generating multisets}\label{ssec:small} We end the discussion of uniqueness for $\mc{S}(R)$ by studying cases with a small cardinality of $R$.

\begin{prop}\label{prop:ap-low-dim-uniqueness}
    For $i = 1, 2$, let $R_i = \{(a_{i,1}, b_{i,1}), \dotsc, (a_{i, k_i}, b_{i, k_i})\}$ be two generating multisets, and for all $j$ we have $b_{i, j} \in [0, a_{i, j})$. Assume that $\mc{S}(R_1 \cup R_1^-) \ab \mc{S}(R_2 \cup R_2^-)$. Then, up to changing $b_{2, j}$ with $a_{2, j} - b_{2, j} \mod a_{2, j}$, and up to swapping indices, we have
    \begin{enumerate}[itemsep=5pt, label*=\arabic*.]
        \item If $k_1 = 1$ and $k_2 = 2$, then either: 
            \begin{enumerate}
                \item $a_{2, 1} = a_{2, 2} = 2a_{1, 1}$. Moreover, $b_{2, 1} = b_{1, 1}$, and $b_{2, 2} = b_{1, 1} + a_{1, 1}$;
                \item $a_{2, 1} = \frac{3}{2} a_{1, 1}$, and $a_{2, 2} = 3 a_{1, 1}$. Moreover, $b_{1, 1} = b_{2, 1} = \frac{a_{1, 1}}{4}$, and $b_{2, 2} = a_{1, 1}$.
            \end{enumerate}

        \item If $k_1 = 1$ and $k_2 = 3$, then either:
            \begin{enumerate}
                \item $a_{2, 1} = a_{2, 2} = a_{2, 3} = 3 a_{1, 1}$. Moreover, $b_{2, 1} = b_{1, 1}$, $b_{2, 2} = b_{1, 1} + a_{1, 1}$, and $b_{2, 3} = b_{1, 1} + 2a_{1, 1}$.
                \item $a_{2, 1} = 2a_{1, 1}$, and $a_{2, 2} = a_{2, 3} = 4a_{1, 1}$. Moreover, either $b_{2, 1} = b_{1, 1}$, $b_{2, 2} = b_{1, 1} + a_{1, 1}$, $b_{2, 3} = b_{1, 1} + 3a_{1, 1}$; or $b_{2, 1} = b_{1, 1} + a_{1, 1}$, $b_{2, 2} = b_{1, 1}$, $b_{2, 3} = b_{1, 1} + 2a_{1, 1}$.
                \item $a_{2, 1} = \frac{3}{2} a_{1, 1}$, and $a_{2, 2} = a_{2, 3} = 6a_{1, 1}$. Moreover, $b_{1, 1} = b_{2, 1} = \frac{a_{1, 1}}{4}$, $b_{2, 2} = \frac{9a_{1, 1}}{4}$, $b_{2, 3} = \frac{21a_{1, 1}}{4}$.
            \end{enumerate}

    \end{enumerate}
\end{prop}
\begin{proof}
    In Items 1, and 2, by re-scaling, we may assume that $a_{1, 1} = 1$ and $b_{1, 1} \in [0, 1)$. By the proof of Lemma \ref{lemma:double-to-single-reduction}, we may assume that $a_{2, j} = \frac{p_{2, j}}{2}$, for some $p_{2, j} \in \mathbb{N}$. Then
    \begin{equation}\label{eq:intermediate}
        \sum_{j = 1}^{k_2} \frac{2}{p_{2, j}} = 1.
    \end{equation}
    Moreover, if one of $p_{2, j}$ is odd, we must have $b_{1, 1} = \frac{1}{4}$ or $b_{1, 1} = \frac{3}{4}$; if this is the case, we may assume $b_{1, 1} = \frac{1}{4}$ without loss of generality.
    \medskip
    
    \emph{Item 1.} The only solutions to \eqref{eq:intermediate} (up to swapping indices) are $p_{2, 1} = p_{2, 2} = 4$, or $p_{2, 1} = 3$, $p_{2, 2} = 6$. In the latter case we have 
    \begin{align*}
        &\left(\mathbb{N} + \frac{1}{4}\right) \cup \left(\mathbb{N} - \frac{1}{4}\right) \ab \left(\frac{3}{2}\mathbb{N} + b_{2, 1}\right) \cup \left(\frac{3}{2}\mathbb{N} - b_{2, 1}\right) \cup \left(3\mathbb{N} + b_{2, 2}\right) \cup \left(3\mathbb{N} - b_{2, 2}\right)\\
        &\ab \left(3\mathbb{N} + b_{2, 1}\right) \cup \left(3\mathbb{N} + \tfrac{3}{2} + b_{2, 1}\right) \cup \left(3\mathbb{N} - b_{2, 1}\right) \cup \left(3\mathbb{N} - \tfrac{3}{2} - b_{2, 1}\right) \cup \left(3\mathbb{N} + b_{2, 2}\right) \cup \left(3\mathbb{N} - b_{2, 2}\right).
    \end{align*}
    It follows that (up to symmetries)
    \[
        b_{2, 1} \equiv \frac{1}{4} \mod \mathbb{Z}, \quad \frac{3}{2} + b_{2, 1} \equiv - \frac{1}{4} \mod \mathbb{Z}, \quad b_{2, 2} \equiv \frac{1}{4} \mod \mathbb{Z},
    \]
    and so $b_{2, 1} \in \{\frac{1}{4}, \frac{5}{4}\}$, $b_{2, 2} \in \{1, 2\}$. 

    In the former case, we have
    \[
        \left(\mathbb{N} + b_{1, 1}\right) \cup \left(\mathbb{N} - b_{1, 1}\right) \ab \left(2\mathbb{N} + b_{2, 1}\right) \cup \left(2\mathbb{N} - b_{2, 1}\right) \cup \left(2\mathbb{N} + b_{2, 2}\right) \cup \left(2\mathbb{N} - b_{2, 2}\right).
    \]
    and it is straightforward to check that the only cases are as claimed above.
    \medskip

    \emph{Item 2.} The only solutions to \eqref{eq:intermediate} (up to swapping indices) are
    \begin{align*}
        (p_1, p_2, p_3) \in \{(3, 7, 42), (3, 8, 24), (3, 9, 18), (3, 10, 15), (3, 12, 12), (4, 5, 20), (4, 6, 12), (4, 8 , 8),\\ 
        \hspace{5cm} (5, 5, 10), (6, 6, 6)\}.
    \end{align*}
    We may discard many of those as follows. Notice first that if we have one or more of $p_i$ odd, we may either re-scale by a factor of two, or replace the arithmetic progressions with half-integer moduli by two arithmetic progressions with twice as large moduli. Thus, by Lemma \ref{lemma:ap-uniqueness}, Item 3, and its proof, we see that we are left with 
    \[
        (p_1, p_2, p_3) \in \{(3, 12, 12), (4, 8 , 8), (6, 6, 6)\}.
    \]
    We are left to consider each case separately. If $(p_1, p_2, p_3) = (3, 12, 12)$, then $b_{1, 1} = \frac{1}{4}$ and
    \begin{align*}
        &\left(\mathbb{N} + \frac{1}{4}\right) \cup \left(\mathbb{N} - \frac{1}{4}\right)\\ 
        &\ab \left(\frac{3}{2}\mathbb{N} + b_{2, 1}\right) \cup \left(\frac{3}{2}\mathbb{N} - b_{2, 1}\right) \cup \left(6\mathbb{N} + b_{2, 2}\right) \cup \left(6\mathbb{N} - b_{2, 2}\right) \cup \left(6\mathbb{N} + b_{2, 3}\right) \cup \left(6\mathbb{N} - b_{2, 3}\right).
    \end{align*}
    Arguing as in Item 1, we see that (up to symmetries) 
    \[
        b_{2, 1} = \frac{1}{4},\quad b_{2, 2} = \frac{9}{4}, \quad b_{2, 3} = \frac{21}{4}.
    \]

    Next, in the case $(p_1, p_2, p_3) = (4, 8, 8)$, we have
    \begin{align*}
        &\left(\mathbb{N} + b_{1, 1}\right) \cup \left(\mathbb{N} - b_{1, 1}\right)\\ 
        &\ab \left(2\mathbb{N} + b_{2, 1}\right) \cup \left(2\mathbb{N} - b_{2, 1}\right) \cup \left(4\mathbb{N} + b_{2, 2}\right) \cup \left(4\mathbb{N} - b_{2, 2}\right) \cup \left(4\mathbb{N} + b_{2, 3}\right) \cup \left(4\mathbb{N} - b_{2, 3}\right).
    \end{align*}
    Therefore we immediately see that either $b_{2, 1} = b_{1, 1}$ or $b_{2, 1} = b_{1, 1} + 1$. In each of the two possibilities, the case study reduces to Item 1, and the result follows.

    Finally, if $(p_1, p_2, p_3) = (6, 6, 6)$, we have
        \begin{align*}
        &\left(\mathbb{N} + b_{1, 1}\right) \cup \left(\mathbb{N} - b_{1, 1}\right)\\ 
        &\ab \left(3\mathbb{N} + b_{2, 1}\right) \cup \left(3\mathbb{N} - b_{2, 1}\right) \cup \left(3\mathbb{N} + b_{2, 2}\right) \cup \left(3\mathbb{N} - b_{2, 2}\right) \cup \left(3\mathbb{N} + b_{2, 3}\right) \cup \left(3\mathbb{N} - b_{2, 3}\right).
    \end{align*}
    We see that (up to symmetries) $b_{2, 1} = b_{1, 1}$, $b_{2, 2} = 1 + b_{1, 1}$, $b_{2, 3} = 2 + b_{1, 1}$. This completes the proof.
\end{proof}

\begin{rem}\rm
    The next natural case in Proposition \ref{prop:ap-low-dim-uniqueness} would be $k_1 = k_2 = 2$. In this situation a similar analysis by hand is possible, but for simplicity of presentation we do not give the details. In general, from Proposition \ref{prop:ap-low-dim-uniqueness} we easily deduce an algorithm which could describe all possible $R_1$ and $R_2$ with $\mc{S}(R_1 \cup R_1^-) \ab \mc{S}(R_2 \cup R_2^-)$ (at least when $k_1 = 1$).
\end{rem}

\subsection{Covering systems.}\label{ssec:covering-systems} Let us introduce some more terminology about arithmetic progressions. We will call a set of pairs of positive integers $(a_i, b_i)$, for $i = 1, \dotsc, k$ a \emph{covering system (CS)}, if for every integer $x \in \mathbb{Z}$ there is at least one $i$ with $x \equiv b_i \mod a_i$. If this happens for exactly one $i$, then the set $(a_i, b_i)$ is called an \emph{exact covering system (ECS)}, and if all $a_i$ are distinct a \emph{distinct covering system (DCS)}. An ECS which appears after dividing integers into a finite number of arithmetic progressions of equal steps, then dividing one of these into further arithmetic progressions, and so on, is called a \emph{natural exact covering system (NECS)}.

These systems are thoroughly studied -- see \cite[Chapter 1]{Soifer-24} for a thorough historical discussion and \cite{Porubsky-81} for a survey, and remain an active field of research to this day \cite{Balister-Bollobas-Morris-Sahasrabudhe-Tiba-22, Filaseta-Ford-Konyagin-Pomerance-Yu-07, Hough-15}. It is well-known that the biggest moduli in an ECS needs to appear at least twice \cite{Porubsky-81}, so in particular there is no ECS which is also a DCS. Also, not all ECS are NECS, see e.g. \cite{Porubsky-76, Schnabel-15}. For a relation of the covering problem with non-vanishing sums of roots of unity, see \cite{Schnabel-15}.

\section{Unique determination}\label{sec:unique-determination}
In this section we discuss how to recover the number and the lengths of boundary components, parallel transport and magnetic flux along boundary, and other invariants of the magnetic DN map from its spectrum. The main analytical input is Theorem \ref{thma}, and we apply our study of arithmetic progressions through close almost bijections developed in Section \ref{sec:ap}. We assume that $(M, g)$ is a compact Riemannian surface, that $E = M \times \mathbb{C}$ is equipped with a purely imaginary $1$-form $A$, and a real-valued potential $q$, as well as that $\partial M$ has $m \geq 1$ boundary components $N_1, \dotsc, N_m$. 

\subsection{One boundary component}\label{ssec:uniqueness-single}
Here we consider the simplest case of a single boundary component, that is, $m = 1$ and extract information from the spectrum of the DN map. We define $p \in \mathbb{Z}$ and $\alpha \in [0, 1)$ such that
\[
    p + \frac{1}{2\pi i} \int_{\partial M} A = \alpha.
\]
We may now establish Theorem \ref{thmb}.
\begin{proof}[Proof of Theorem \ref{thmb}]
    Enumerate the spectrum of $\Lambda_{g, A, q}$ in the non-decreasing order by $(\sigma_n)_{n \in \mathbb{N}}$. We split the discussion to cases depending on the value of $\alpha$. 
    \medskip
    
    \emph{Case 1: $\alpha \in (0, \frac{1}{2})$.} Then Theorem \ref{thma} implies that for $n$ large enough
    \begin{align*}
        \sigma_{2n} &= \frac{2\pi}{\ell(\partial M)} \left(n - \alpha\right) + n^{-1} \left(-\frac{i}{4\pi} \int_{\partial M} \iota_{\nu} dA + \frac{1}{4\pi} \int_{\partial M} q\right) + \mc{O}(n^{-2}),\\
        \sigma_{2n + 1} &= \frac{2\pi}{\ell(\partial M)} \left(n + \alpha\right) + n^{-1} \left(\frac{i}{4\pi} \int_{\partial M} \iota_{\nu} dA + \frac{1}{4\pi} \int_{\partial M} q\right) + \mc{O}(n^{-2}),
    \end{align*}
    as $n \to \infty$. We therefore see that
    \begin{align}\label{eq:uniqueness-limits}
    \begin{split}
        &\lim_{n \to \infty} \frac{\sigma_{2n}}{n} = \frac{2\pi}{\ell(\partial M)}, \quad \lim_{n \to \infty} \left(\sigma_{2n} - \frac{2\pi}{\ell(\partial M)}n\right) = -\frac{2\pi}{\ell(\partial M)}\alpha,\\ 
        &\lim_{n \to \infty} n \left(\sigma_{2n} - \frac{2\pi}{\ell(\partial M)}(n - \alpha)\right) = -\frac{i}{4\pi} \int_{\partial M} \iota_\nu dA + \frac{1}{4\pi} \int_{\partial M} q,\\
        &\lim_{n \to \infty} n \left(\sigma_{2n + 1} - \frac{2\pi}{\ell(\partial M)}(n  + \alpha) \right) = \frac{i}{4\pi} \int_{\partial M} \iota_\nu dA + \frac{1}{4\pi} \int_{\partial M} q.
    \end{split}
    \end{align}
    The first limit shows that the spectrum determines $\ell(\partial M)$, the second one that it determines $\alpha$, while the last two after summing and subtracting determine $\int_{\partial M} \iota_\nu dA$ and $\int_{\partial M} q$.
    \medskip

    \emph{Case 2: $\alpha \in (\frac{1}{2}, 1)$.} By the same reasoning as in Case 1, we have for $n$ large enough
    \begin{align*}
        \sigma_{2n} &= \frac{2\pi}{\ell(\partial M)} \left(n - 1 + \alpha + n^{-1} \left(\frac{i}{4\pi} \int_{N_j} \iota_{\nu} dA + \frac{1}{4\pi} \int_{\partial M} q\right)\right) + \mc{O}(n^{-2}),\\
        \sigma_{2n + 1} &= \frac{2\pi}{\ell(\partial M)} \left(n + 1 - \alpha + n^{-1} \left(-\frac{i}{4\pi} \int_{\partial M} \iota_{\nu} dA + \frac{1}{4\pi} \int_{\partial M} q\right)\right) + \mc{O}(n^{-2}),
    \end{align*}
    as $n \to \infty$. Applying the same arguments as in Case 1, we see that we may determine $1 - \alpha$ instead of $\alpha$ and $-\int_{\partial M} \iota_\nu dA$ instead of $\int_{\partial M} \iota_\nu dA$ in the same fashion. (Alternatively, we could have obtained these formulas by using that $\Lambda_{g, A, q}$ and $\Lambda_{g, -A, q}$ have the same spectrum according to Proposition \ref{prop:symmetry}, and then applying Step 1.)
    \medskip

    \emph{Case 3: $\alpha = 0$.} In the case $i \int \iota_\nu dA > 0$, we similarly to above cases have for $n$ large enough that
    \begin{align*}
        \sigma_{2n} &= \frac{2\pi}{\ell(\partial M)} \left(n + n^{-1} \left(-\frac{i}{4\pi} \int_{\partial M} \iota_{\nu} dA + \frac{1}{4\pi} \int_{\partial M} q\right)\right) + \mc{O}(n^{-2}),\\
        \sigma_{2n + 1} &= \frac{2\pi}{\ell(\partial M)} \left(n + n^{-1} \left(\frac{i}{4\pi} \int_{\partial M} \iota_{\nu} dA + \frac{1}{4\pi} \int_{\partial M} q\right)\right) + \mc{O}(n^{-2}),
    \end{align*}
    as $n \to \infty$, and using the formulas as in Case 1, we may uniquely determine the claimed quantities. 

    In the case $i \int \iota_\nu dA < 0$ we get
    \begin{align*}
        \sigma_{2n} &= \frac{2\pi}{\ell(\partial M)} \left(n + n^{-1} \left(\frac{i}{4\pi} \int_{\partial M} \iota_{\nu} dA + \frac{1}{4\pi} \int_{\partial M} q\right)\right) + \mc{O}(n^{-2}),\\
        \sigma_{2n + 1} &= \frac{2\pi}{\ell(\partial M)} \left(n + n^{-1} \left(-\frac{i}{4\pi} \int_{\partial M} \iota_{\nu} dA + \frac{1}{4\pi} \int_{\partial M} q\right)\right) + \mc{O}(n^{-2}),
    \end{align*}
    and the same argument applies.

    Finally, if $\int \iota_\nu dA = 0$, then it is possible to split as in the two preceding scenarios according to whether $\int_{\partial M} q$ is positive or negative. If $\int_{\partial M} q = 0$, then we see that
    \begin{align*}
        \sigma_{2n} &= \frac{2\pi}{\ell(\partial M)}n + \mc{O}(n^{-2}),\\
        \sigma_{2n + 1} &= \frac{2\pi}{\ell(\partial M)}n + \mc{O}(n^{-2}),
    \end{align*}
    so all the limits in \eqref{eq:uniqueness-limits} are zero.
    \medskip

    \emph{Case 4: $\alpha = \frac{1}{2}$.} The discussion in this case is analogous to Case 3, and we omit the details.
    \medskip
    
    Combining all the above cases, this implies the claim of the proposition and completes the proof.
\end{proof}

\begin{rem}\rm
    That we can recover only $e^{\pm \int_{\partial M} A}$ and $\left|\int_{\partial M} \iota_{\nu} dA\right|$ is actually the most we can do from the first three terms. Indeed, by Proposition \ref{prop:symmetry}, we have $\Lambda_{g, -A, Q} = C \Lambda_{g, A, q} C$ and so the Steklov spectrum cannot distinguish between $A$ and $-A$.
\end{rem}

\smallskip

\subsection{Uniqueness: general results}\label{ssec:uniqueness-general} We start with the following lemma about the spectrum of the magnetic DN map. Let $(M, g)$ be a compact Riemannian surface with $m$ boundary components $N_1, \dotsc, N_m$, let $A$ be a purely imaginary $1$-form, and $q$ a real-valued potential, and let $(p_j)_{j = 1}^m \in \mathbb{Z}^m$. We recall the notation
    \[  
        R^\pm(M, g, A, q) := \bigcup_{j = 1}^{m} \left\{\left(\frac{2\pi}{\ell(N_j)}, \pm\frac{2\pi}{\ell(N_j)}\left(p_j + \frac{1}{2\pi i} \int_{N_j} A\right)\right)\right\}.
    \]
We note that $(p_j)_{j = 1}^m$ will be implicit in our notation and its value clear from the context.    
\begin{lemma}\label{lemma:spectrum-ap-close-bijection}
    We use the same notation and assumptions as in Theorem \ref{thma}. Then, there exists a close almost bijection $F$
    \[
        F: \mc{S}(R^+ \cup R^-) \to \Spec(\Lambda_{g, A, q}).
    \]
\end{lemma}
\begin{proof}
    For $n \in \mathbb{N}$, $j = 1, \dotsc, m$, we define $F$ as
    \[
        F\left(\frac{2\pi}{\ell(N_j)} n \pm\frac{2\pi}{\ell(N_i)}\left(p_j + \frac{1}{2\pi i} \int_{N_j} A\right)\right) := \lambda_{\pm n}^{(j)},
    \]
    where we recall $(\lambda_n^{(j)})_{n \in \mathbb{Z}}$ is a suitable enumeration of the set $\mc{S}_j$. It is clear that $F$ is well-defined and that it is an almost bijection since $\Spec(\Lambda_{g, A, q}) = \cup_{j = 1}^m \mc{S}_j$. Finally, it is also close by the asymptotics result of Theorem \ref{thma}. This completes the proof.
\end{proof}

We now state our first general unique determination result, which is the first part of Theorem \ref{thmc}.

\begin{theorem}\label{thm:first-general}
    For $i = 1, 2$, let $(M_i, g_i)$ be compact Riemannian surfaces, let $A_i$ be purely imaginary $1$-forms, and let $q_i$ be real potentials, and write $R^\pm_i := R^\pm(M_i, g_i, A_i, q_i)$. Then, there exists a close almost bijection 
    \[
        F: \Spec(\Lambda_{g_1, A_1, q_1}) \to \Spec(\Lambda_{g_2, A_2, q_2})
    \]
    if and only if $\mc{S}(R^+_1 \cup R^-_1) \ab \mc{S}(R^+_2 \cup R^-_2)$.
\end{theorem}
\begin{proof}
    By Lemma \ref{lemma:spectrum-ap-close-bijection}, for $i = 1, 2$ there are close almost bijections $F_i: \mc{S}(R^+_i \cup R^-_i) \to \Spec(\Lambda_{g_i, A_i, q_i})$. Thus if $\mc{S}(R^+_1 \cup R^-_1) \ab \mc{S}(R^+_2 \cup R^-_2)$, or in other words the identity map (possibly changed at finitely many points) is a close almost bijection between $\mc{S}(R^+_1 \cup R^-_1)$ and $\mc{S}(R^+_2 \cup R^-_2)$, by composition we get a close almost bijection $F$ as in the statement (equal to $F_2 F_1^{-1}$ for all but finitely many points).
        
    Conversely, assume that $F: \Spec(\Lambda_{g_1, A_1, q_1}) \to \Spec(\Lambda_{g_2, A_2, q_2})$ is a close almost bijection. Again, by composition there is a close almost bijection $\widetilde{F}: \mc{S}(R^+_1 \cup R^-_1) \to \mc{S}(R^+_2 \cup R^-_2)$ (equal to $F_2^{-1} F F_1$ for all but finitely many points). Applying Lemma \ref{lemma:acb-implies-equal} immediately completes the proof.
\end{proof}

We now consider the next simplest case, where we assume that one surface has a single boundary component.

\begin{prop}\label{prop:ECS}
    We use the notation of Theorem \ref{thm:first-general}, where we additionally assume that $M_1$ and $M_2$ have $k_1 = 1$ and $k_2$ boundary components, $N_{1, 1} = \partial M_1$, and $(N_{2, j})_{j = 1}^{k_2}$, respectively. Let $p_{1, 1}, p_{2, 1}, \dotsc, p_{2, k_2} \in \mathbb{Z}$ be integers such that
    \[
        \alpha_{1, 1} := p_{1, 1} + \frac{1}{2\pi i} \int_{\partial N_{1, 1}} A_1 \in [0, 1), \quad \alpha_{2, j} := p_{2, j} + \frac{1}{2\pi i} \int_{N_{2, j}} A_2 \in [0, 1), \quad j = 1, \dotsc, k_2. 
    \]
    Then, if $\alpha_{1, 1} \neq 0, \frac{1}{4}, \frac{1}{2}, \frac{3}{4}$, there is a close almost bijection between the spectra $\Spec(\Lambda_{g_1, A_1, q_1})$ and $\Spec(\Lambda_{g_2, A_2, q_2})$ if and only if the pairs
    \begin{equation}\label{eq:claimECS}
        \Big(\frac{\ell(N_{1, 1})}{\ell(N_{2, j})}, \frac{\ell(N_{1, 1})}{\ell(N_{2, j})} \varepsilon_j \alpha_{2, j} - \alpha_{1, 1}\Big), \qquad \quad j = 1, \dotsc, k_2,
    \end{equation}
    form an exact covering system, for some choice of signs $\varepsilon_j \in \{\pm 1\}$.
\end{prop}
\begin{proof}
    By Theorem \ref{thm:first-general}, such a close almost bijection exists if and only if $\mc{S}(R^+_1 \cup R^-_1) \ab \mc{S}(R^+_2 \cup R^-_2)$. By re-scaling, this in turn equivalent to
    \begin{equation}\label{eq:equivalent-to}
        (\mathbb{N} + \alpha_{1,1}) \cup  (\mathbb{N} - \alpha_{1,1}) \ab \bigcup_{j = 1}^{k_2} \left(\frac{\ell(N_{1, 1})}{\ell(N_{2, j})}\left(\mathbb{N} + \alpha_{2, j}\right) \cup \frac{\ell(N_{1, 1})}{\ell(N_{2, j})}\left(\mathbb{N} - \alpha_{2, j}\right)\right).
    \end{equation}
    The claim now follows immediately from Lemma \ref{lemma:double-to-single-reduction}.
\end{proof}

\subsection{Uniqueness: small number of boundary components}\label{ssec:small-boundary}
We end this section with a discussion of examples.

\begin{prop}\label{prop:rigid}
    Assume the notation of Theorem \ref{thm:first-general}, and for $i = 1, 2$, write $N_{i, j}$, $j = 1, \dotsc k_i$, for the boundary components of $\partial M_i$, and let $p_{i, j} \in \mathbb{Z}$ be such that
    \[
        \alpha_{i, j} := p_{i, j} + \frac{1}{2\pi i} \int_{\partial N_{i, j}} A_i \in [0, 1).
    \]
    We have $\Spec(\Lambda_{g_1, A_1, q_1})$ and $\Spec(\Lambda_{g_2, A_2, q_2})$ in a close almost bijection, in the following situations:
    \begin{enumerate}[itemsep = 5pt, label*=\arabic*.]
        \item If $k_1 = 1$ and $k_2 = 2$, if and only if either (up to symmetries)
        \begin{enumerate}[itemsep=5pt]
            \item $\ell(N_{2, 1}) = \ell(N_{2, 2}) = \frac{1}{2} \ell(\partial M_1)$, and $\alpha_{2, 1} = \frac{1}{2} \alpha_{1, 1}$, $\alpha_{2, 2} = \frac{1}{2}(1 + \alpha_{1, 1})$;
            \item $\ell(N_{2, 1}) = \frac{2}{3}\ell(\partial M_1)$, $\ell(N_{2, 2}) = \frac{1}{3} \ell(\partial M_1)$, and $\alpha_{1, 1} = \frac{1}{4}$, $\alpha_{2, 1} = \frac{1}{6}$, $\alpha_{2, 2} = \frac{1}{3}$.
        \end{enumerate}
        \item If $k_1 = 1$ and $k_2 = 3$, if and only if either (up to symmetries) 
        \begin{enumerate}[itemsep=5pt]
            \item $\ell(N_{2, 1}) = \ell(N_{2, 2}) = \ell(N_{2, 3}) = \frac{1}{3} \ell(\partial M_1)$, and 
            \[
                \alpha_{2, 1} = \frac{1}{3} \alpha_{1, 1},\quad \alpha_{2, 2} = \frac{1}{3}(1 + \alpha_{1, 1}),\quad \alpha_{2, 3} = \frac{1}{3}(2 + \alpha_{1, 1}).
            \]
            
            \item $\ell(N_{2, 1}) = \frac{1}{2}\ell(\partial M_1)$, $\ell(N_{2, 2}) = \ell(N_{2, 3}) = \frac{1}{4} \ell(\partial M_1)$, and either 
            \[
                \alpha_{2, 1} = \frac{1}{2}\alpha_{1, 1},\quad \alpha_{2, 2} = \frac{1}{4}(1 + \alpha_{1, 1}),\quad \alpha_{2, 3} = \frac{1}{4}(1 + 3\alpha_{1, 1}),
            \]
            or 
            \[ 
                \alpha_{2, 1} = \frac{1}{2}(1 + \alpha_{1, 1}), \quad \alpha_{2, 2} = \frac{1}{4}\alpha_{1, 1},\quad \alpha_{2, 3} = \frac{1}{4}(1 + 2\alpha_{1, 1}).
            \]

            \item $\ell(N_{2, 1}) = \frac{2}{3}\ell(\partial M_1)$, $\ell(N_{2, 2}) = \ell(N_{2, 3}) = \frac{1}{6} \ell(\partial M_1)$, and 
            \[
                \alpha_{1, 1} = \frac{1}{4}, \quad \alpha_{2, 1} = \frac{1}{6},\quad \alpha_{2, 2} = \frac{3}{8},\quad \alpha_{2, 3} = \frac{7}{8}.
            \]
        \end{enumerate}

    \end{enumerate}
\end{prop}
Here, by ``up to symmetries" we mean up to changing $\alpha_{i, j}$ by $1 - \alpha_{i, j} \mod \mathbb{Z}$, or replacing $\alpha_{i, j}$ by $\alpha_{i', j'}$ when $\ell(N_{i, j}) = \ell(N_{i', j'})$, for some indices $i, j, i', j'$.
\begin{proof}
    As in the previous arguments, by Theorem \ref{thm:first-general}, the two spectra are in close almost bijection if and only if $\mc{S}(R_1^+ \cup R_1^-) \ab \mc{S}(R_2^+ \cup R_2^-)$. For all $i, j$, note that $\ell(N_{i, j}) = \frac{2\pi}{a_{i, j}}$, $\alpha_{i, j} = \frac{\ell(N_{i, j})}{2\pi} b_{i, j}$, and $N_{1, 1} = \partial M_1$. After an easy calculation, the result is an immediate consequence of Proposition \ref{prop:ap-low-dim-uniqueness}.
\end{proof}

We end this subsection with a statement concerning closed $1$-forms, i.e. zero magnetic field, in which case the rigidity statements are stronger.

\begin{prop}\label{prop:rigid-flat}
    Assume the notation of Proposition \ref{prop:rigid}, and assume furthermore that $dA_1 = dA_2 = 0$, and $k_1 = 1$. Then $\Spec(\Lambda_{g_1, A_1, q_1})$ and $\Spec(\Lambda_{g_2, A_2, q_2})$ cannot be in almost close bijection if $k_2 = 2$. If $k_2 = 3$, this happens if and only if $\alpha_{1, 1} = 0$, and $\alpha_{2, 1} = 0$, $\alpha_{2, 2} = \frac{1}{3}$, $\alpha_{2, 3} = \frac{2}{3}$ (up to symmetries).
\end{prop}
We remark that in the second situation, under the assumption that every closed curve is homotopic to a boundary curve (fulfilled e.g. for planar domains), specifying $(\alpha_{2, j})_{j = 1}^{k_2}$ determines the gauge equivalence class of $A_2$, i.e. it determines $A_2$ up to factors of $f^{-1}df$, where $f: M_2 \to \mathbb{S}^1 \subset \mathbb{C}$. Thus the proposition determines (up to symmetries) the gauge class of $A_2$ when $k_2 = 3$.
\begin{proof}
    By Stokes' theorem, for $i = 1, 2$ we have 
    \[
        0 = \int_{M_i} dA_i = \sum_{j = 1}^{k_i} \int_{N_{i, j}} A_i \equiv \sum_{j = 1}^{k_i} \alpha_{i, j} \mod \mathbb{Z}.
    \]
    Thus $\alpha_{1, 1} = 0$, and after going through the cases of Proposition \ref{prop:rigid}, the result immediately follows.
\end{proof}

\subsection{An example}\label{ssec:example-up-to-smoothing} In Theorem \ref{thmc}, Items 1 and 2, we considered the spectrum up to the equivalence relation of close almost bijection. It is natural to ask whether one can recover further terms in the spectral expansions, but also whether it is possible to construct examples in which the spectra agree up to $\mc{O}(n^{-\infty})$. Here we do the latter.

\begin{prop}\label{prop:example}
    We use the notation of Theorem \ref{thm:first-general}. Assume that $q_1 \equiv q_2 \equiv 0$, and that near the the boundary components in boundary normal coordinates, $A_1$ and $A_2$ are zero in the normal direction, and are \emph{constant} in the directions parallel to the boundary. Assume further that $\mc{S}(R_1^+ \cup R_1^-) \ab \mc{S}(R_2^+ \cup R_2^-)$. Then
    \[
        \sigma_n(M_1, g_1, A_1) - \sigma_n(M_2, g_2, A_2) = \mc{O}(n^{-\infty}), \quad n \to \infty.
    \]
\end{prop}
Here, for $i = 1, 2$, $\sigma_n(M_i, g_i, A_i)$ denotes the $n$th eigenvalue of $\Lambda_{g_i, A_i, 0}$. The assumption that $A_i$ are constant near the boundary can be relaxed to the assumption of constant up to infinite order on the boundary.
\begin{proof}
    The crucial observation is that by Proposition \ref{prop:computation}, and its proof, the symbol of the magnetic DN map is precisely $\sqrt{g^{11}} |\xi_1| - i\sqrt{g^{11}} A_1 \frac{\xi_1}{|\xi_1|}$. Indeed, firstly, the $a_2$ part of the symbol is zero by looking at the formula in Proposition \ref{prop:computation}. Thus, the symbol of the operator $B$ appearing in Proposition \ref{prop:computation} satisfies $b_{-1} = 0$, and we see directly from \cite[Equation 3.8]{Cekic-20} that the symbols $b_{-2}, b_{-3}, \dotsc$ have to vanish, which gives the desired expansion.

    Therefore $\Lambda_{g_i, A_i, 0}$ already has an $x$-independent symbol and we may directly apply Theorem \ref{thm:sharpasymptotics}, which concludes the proof.
\end{proof}
\bibliographystyle{alpha}
\bibliography{biblio}

\end{document}